\documentclass[11pt,a4paper]{article}
\usepackage{amsmath}
\usepackage{amssymb}
\usepackage{amsthm}
\usepackage[margin=2.1cm]{geometry}
\usepackage{color}
\usepackage{microtype}
\usepackage{graphicx}
\usepackage{standalone}
\usepackage{tikz}
\usepackage{enumerate}
\usepackage{mathrsfs}

\newtheorem{thm}{Theorem}

\newtheorem{lem}{Lemma}

\theoremstyle{remark}
\newtheorem*{rmrk}{Remark}
\usepackage{color}
\definecolor{comment}{RGB}{65,104,23}

\newcommand{\pert}[1] {\delta_{#1}\hspace{-0.5pt}}

\newcommand{\Lie}[1]{\mathcal{L}_{#1}}
\newcommand{\I}{{\mspace{1mu}\mathrm{I}}} 

\newcommand{\at}{\makeatletter @\makeatother}

\makeatletter
\renewcommand{\maketitle}{\bgroup\setlength{\parindent}{0pt}
\begin{flushleft}
  \textbf{\Large\@title}\par
	\vspace{.3cm}
  \@author\par
\noindent\rule[0.5ex]{\linewidth}{0.5pt}
\end{flushleft}\egroup
}
\makeatother

\title{On the well-posedness of Galbrun's equation}
\author{\underline{{LINUS H{\"A}GG} and MARTIN BERGGREN}\\
Department of Computing Science,
Ume{\aa} University, SE--901 87 Ume{\aa}, Sweden.\\
\{linush, martin.berggren\}\raisebox{-1pt}{{\at}}cs.umu.se\\
\today}

\begin{document}
\thispagestyle{empty}
\maketitle

\abstract{Galbrun's equation, which is a second order partial differential equation describing the evolution of a so-called Lagrangian displacement vector field, can be used to study acoustics in background flows as well as perturbations of astrophysical flows.      
Our starting point for deriving Galbrun's equation is linearized Euler's equations, which is a first order system of partial differential equations that describe the evolution of the so-called Eulerian flow perturbations. 
Given a solution to linearized Euler's equations, we introduce the Lagrangian displacement as the solution to a linear first order partial differential equation, driven by the Eulerian perturbation of the fluid velocity.
Our Lagrangian displacement solves Galbrun's equation, provided it is regular enough and that the so-called ``no resonance'' assumption holds. 
In the case that the background flow is steady and tangential to the domain boundary, we prove existence, uniqueness, and continuous dependence on data of solutions to an initial--boundary value problem for linearized Euler's equations. 
For such background flows, we demonstrate that the Lagrangian displacement is well-defined, that the initial datum of the Lagrangian displacement can be chosen in order to fulfill the ``no resonance'' assumption, and derive a classical energy estimate for (sufficiently regular solutions to) Galbrun's equation.
Due to the presence of zeroth order terms of indefinite signs in the equations, the energy estimate allows solutions that grow exponentially with time.}

\section{Introduction}
The linearized Euler's equations constitute a standard model for propagation of sound in a background flow. 
What appears to be less known is that the linearized Euler's equations can be reduced to a vector ``wave'' equation in the Lagrangian displacement, that is, the displacement of individual fluid particles. (The precise definition of the Lagrangian displacement will be given later.) 
The resulting equation is often referred to as Galbrun's equation in the literature, in honor of Henri Galbrun who first derived the equation in 1931~\cite{Ga31}[Chapter~3]. 
Since that first account, Galbrun's equation (or at least very similar equations) has been independently rediscovered and investigated a multiple of times~\cite{FrRo60,Ha70,Go97,LyOs67, FrSc78a, FrSc78b, DySc79}, with applications in acoustics and astrophysics.

The linearized Euler's equations are derived from Euler's equations by using an \emph{Eulerian} linearization ansatz~\cite{Ga03}. 
Analogously, Galbrun's equation may be derived by using a \emph{Lagrangian} linearization ansatz~\cite{Ga03}.
However, we will present a complementary derivation, for homentropic background flows, of Galbrun's equation  that does not rely on Lagrangian perturbations.  
One of many possible formulations of Galbrun's equation reads~\cite{Ga03,Br11,Fe13} (note that the last two references assume that $\varphi_0 = \pert{}\varphi = 0$)
\begin{equation}
\rho_0D_0^2w-\nabla(\rho_0c_0^2\nabla\cdot w)+(\nabla p_0)\nabla\cdot w - (\nabla w)^T\nabla p_0 -\rho_0(w\cdot\nabla)\varphi_0 = \rho_0\,\pert{}\varphi,
\label{eq:GStandardForm}
\end{equation}
where the vector field $w$ denotes the Lagrangian displacement; $u_0$, $p_0$, $\rho_0$, and $c_0$ the fluid velocity, pressure, density, and speed of sound fields of the background flow, respectively; $\varphi_0$ the volume force density acting on the background fluid; $\pert{}\varphi$ a volume force density; and $D_0 = \partial_t + u_0\cdot\nabla$ the material derivative with respect to $u_0$.
Apart from reducing the number of unknowns and equations to be solved, it is pointed out in the literature that Galbrun's equation allows natural handling of boundary conditions, since the primary unknown is the Lagrangian displacement~\cite{Go97, MeBoMiPePe12}. 
Moreover, Galbrun's equation~\eqref{eq:GStandardForm} may be derived using the Euler--Lagrange formalism and thus allows formulation of a wave energy balance law~\cite{Ha70,FrSc78a,Go97,Br11}. 

\subsection{Previous works on the well-posedness of Galbrun's equation} 
Naive finite element discretizations of the time harmonic counterpart of Galbrun's equation~\eqref{eq:GStandardForm} for steady background flows are known to yield poor numerical results~\cite{TrGaBe03,Le03}. 
The situation here appears similar to the issues of locking in linear elasticity and approximation of the curl-curl operator in electrodynamics. 

In a sequence of papers that consider increasingly complicated background flows, a regularized formulation of Galbrun's equation has been proposed to resolve the numerical issues for subsonic background flows~\cite{BoLeLu01,BoLeLu03,BoDuLeMe07,BoMeMiPe10,BoMeMiPePe12}.
The idea behind the regularization is to take advantage of the identity $-\Delta w = -\nabla(\nabla\cdot w) + \nabla\times(\nabla\times w)$.
For a homogenous background flow, the time harmonic Galbrun's equation at angular frequency $\omega$ reads
\begin{equation}
\hat{D}_0^2 \hat{w} - c_0^2\nabla(\nabla\cdot \hat{w})  = \pert{}\hat{\varphi},
\label{eq:GSimple}
\end{equation}
where $\hat{D}_0 = i\omega + u_0\cdot\nabla$, $w(x,t) = \hat{w}(x)\exp{i\omega t}$, and $\pert{}\varphi(x,t) = \pert{}\hat{\varphi}(x)\exp{i\omega t}$.
The regularized formulation of equation~\eqref{eq:GSimple} is constructed by adding $c_0^2\nabla\times(\nabla\times \hat{w} - \psi)$ to the left hand side of the equation.  
We note that if $\psi = \nabla\times \hat{w}$, the added term vanishes and the time harmonic Galbrun's equation~\eqref{eq:GSimple} is retrieved.
The regularized equation is coupled with an equation for $\psi$,
\begin{equation}
\hat{D}_0^2\psi - \underbrace{[\hat{D}_0^2,\nabla\times]\hat{w}}_{=\,0} = \nabla\times \pert{}\hat{\varphi},
\label{eq:curlG}
\end{equation}   
which is obtained by applying the curl operator to the time harmonic Galbrun's equation~\eqref{eq:GSimple} and in the end replacing $\nabla\times w$ with $\psi$.
Although the commutator term $[(i\omega+u_0\cdot\nabla)^2,\nabla\times]\hat{w}$ in equation~\eqref{eq:curlG} vanishes for a homogeneous background flow, we note that it would be present for more complicated background flows.
In order to reconcile the regularized formulation with the original formulation, it is required that $\psi = \nabla\times \hat{w}$ on the boundary of the domain. 
It turns out that the regularized formulation of time harmonic Galbrun's equation, with perfectly matched layers to handle artificial boundaries, is well-posed in two spatial dimensions under relatively mild assumptions on the background flow~\cite{BoMeMiPePe12}.
Nevertheless, a recent numerical study~\cite{ChDu18} reported lack of convergence of the numerical solution in the case of a heterogeneous background flow; interestingly no such convergence issues were observed when solving linearized Euler's equations.   
 
An alternative to the regularized time-harmonic formulation that has been used to generate numerical solutions~\cite{TrGaBe03}, is based on a mixed variational formulation of the system
\begin{subequations}
\begin{align}
\rho_0\hat{D}_0^2\hat{w}+\nabla(\pert{L}\hat{p})+(\nabla p_0)\nabla\cdot \hat{w} - (\nabla \hat{w})^T\nabla p_0 -\rho_0(\hat{w}\cdot\nabla)\varphi_0  &= \rho_0\,\pert{}\hat{\varphi}\\
\pert{L}\hat{p} + \rho_0c_0^2 \nabla\cdot \hat{w} & = 0,
\end{align}
\label{eq:mixed}%
\end{subequations}
where $\pert{L}p(x,t) = \pert{L}\hat{p}(x)\exp{i\omega t}$ denotes the Lagrangian pressure perturbation (a precise definition of Lagrangian perturbations are given in the next section). 
To the best of our knowledge, well-posedness of formulation~\eqref{eq:mixed} has not been established.
  
The case of general time dependence appears to have received less attention in the literature than its harmonic counterpart.
For homogeneous background flows, a regularized formulation, analogous to that used in the time-harmonic case, is known to be well-posed in two spatial dimensions~\cite{Be06,BoBeJo06,BeBoJo06}.
Similarly as in the time-harmonic case, numerical experiments demonstrate that naive discretizations yield poor approximations~\cite{Be06,BoBeJo06,BeBoJo06}.  
We note that the system formulation, analogous to formulation~\eqref{eq:mixed}, has also been studied for general time dependence~\cite{Fe13,FeBeBa16}. 
To the best of our knowledge, well-posedness of such formulation has not been proven.          
		
\section{Derivation of Galbrun's equation from Euler's equations}
\label{sec:derivation}
{\noindent}In this section, we derive Galbrun's equation from Euler's equations via the linearized Euler's equations. We consider an inviscid fluid that either undergoes homentropic flow (the entropy is constant in time and space) or is elastic (the equation of state is independent of the entropy).  The time evolution of the flow is assumed to be governed by Euler's equations
\begin{subequations}
\begin{align}
	\rho Du + \nabla p  &= \rho\varphi,\label{eq:Eu1}\\
	D\rho+\rho\nabla\cdot u &= 0,\label{eq:Eu2}\\
	p &= \Sigma(\rho),\label{eq:Eu3}
\end{align}
\label{eq:Eu}%
\end{subequations}
where $u$, $p$, $\rho$, and $\varphi$ denote the fluid velocity, pressure, density, and volume force density fields, respectively, and $D = \partial_t + u\cdot\nabla$ the material derivative. 
Equations~\eqref{eq:Eu1} and~\eqref{eq:Eu2} express conservation of momentum and mass, respectively, while relation~\eqref{eq:Eu3} is called the (homentropic) equation of state, which upon differentiation gives the speed of sound $c = \sqrt{\Sigma'(\rho)}$.  

With the intention to study the evolution of small perturbations of the flow, we introduce the linearization ansatz 
\begin{equation}
\phi(x,t) = \phi_0 (x,t)+\pert{}\phi(x,t),
\label{eq:pertEu}
\end{equation}
where $\phi(x,t)$ denotes a generic flow field, $\phi_0(x,t)$ is given, and $\pert{}\phi(x,t)$ denotes the so-called Eulerian perturbation.
As before, the given fields $u_0, p_0, \rho_0$, and $\varphi_0$ are termed the background flow, and we require that they themselves satisfy Euler's equations, that is,  
\begin{subequations}
\begin{align}
	\rho_0 D_0u_0 + \nabla p_0  &= \rho_0\varphi_0,\label{eq:Eu01}\\
	D_0\rho_0+\rho_0\nabla\cdot u_0 &= 0,\label{eq:Eu02}\\
	p_0 &= \Sigma(\rho_0).\label{eq:Eu03}
\end{align}
\label{eq:Eu0}%
\end{subequations} 
Substituting the linearization ansatz~\eqref{eq:pertEu} into Euler's equations~\eqref{eq:Eu} and retaining terms that are at most linear in the perturbation, we obtain the linearized Euler's equations
{\begin{subequations}
\begin{align}
\rho_0{D}_0\,\pert{}u + \nabla\pert{}p + \rho_0(\pert{}u\cdot\nabla){u_0}-\frac{\nabla p_0}{\rho_0}\pert{}\rho &= \rho_0\,\pert{}{\varphi},\label{eq:ElE1}\\
{D}_0\,\pert{}\rho +\rho_0\nabla\cdot\pert{}u + (\pert{}u\cdot\nabla)\rho_0+(\nabla\cdot u_0)\pert{}\rho&= 0,\label{eq:ElE2}\\
\pert{}p &= c_0^2\,\pert{}\rho,\label{eq:ElE3}
\end{align}
\label{eq:ElE}%
\end{subequations}}
which describe the evolution of Eulerian perturbations.

Informally, the Lagrangian displacement $w$ is the displacement of individual fluid particles, as illustrated in Figure~\ref{fig:w}.
A more precise definition of the Lagrangian displacement as the displacement of individual fluid particles is given in Appendix~\ref{sec:LagDisp}.   
\begin{figure}[ht!!!]%
\centering
\includegraphics[trim={10mm 0mm 0mm 20mm},clip]{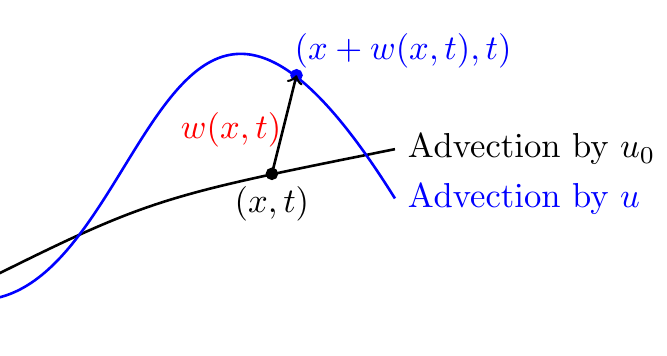}%
\caption{The Lagrangian displacement $w$.}%
\label{fig:w}%
\end{figure}
As detailed by Gabard~\cite{Ga03}, Galbrun's equation may be derived through a Lagrangian linearization of Euler's equations~\eqref{eq:Eu}, that is, linearizing the equations using the ansatz $\phi(x,t) = \phi_0(x,t) + \pert{L}\phi(x,t)$, where the Lagrangian perturbation is given by 
\begin{equation}
\pert{L}\phi(x,t) = \pert{}\phi(x,t) + (w(x,t)\cdot\nabla)\phi_0(x,t) \approx \phi(x+w(x,t),t) - \phi_0(x,t).
\label{eq:dL}
\end{equation}
We will, however, derive Galbrun's equation directly from linearized Euler's equations~\eqref{eq:ElE}.
Given $\pert{}u$ from linearized Euler's equations~\eqref{eq:ElE}, we introduce the Lagrangian displacement $w$ abstractly by the relation
\begin{equation}
\underbrace{(\partial_t + \mathcal{L}_{u_0}){w}}_{=D_0w-(w\cdot\nabla)u_0} \!\!= \pert{}u,\label{eq:dEuRelw}
\end{equation}
where $\Lie{u_0}w = (u_0\cdot\nabla)w-(w\cdot\nabla)u_0$ denotes the Lie derivative of $w$ along $u_0$.
As will be discussed in the sequel, suitable initial and boundary conditions need to be supplied to equation~\eqref{eq:dEuRelw} in order to make $w$ well-defined.   
That $w$ satisfying relation~\eqref{eq:dEuRelw} indeed qualifies as a Lagrangian displacement is motivated in Appendix~\ref{sec:LagDisp}.

The usefulness of definition~\eqref{eq:dEuRelw} stems from the identity
\begin{equation}
\nabla\cdot(\rho_0\,\pert{}u) = \rho_0D_0\left(\frac{1}{\rho_0}\nabla\cdot(\rho_0 w)\right)\!\!,
\label{eq:divdEu}
\end{equation} 
which we now demonstrate. 
By product rule~\eqref{eq:LieProductScalarAndVector}, we obtain from definition~\eqref{eq:dEuRelw} that
\begin{equation}
\rho_0\,\pert{}u = \rho_0(\partial_t + \Lie{u_0})w = (\partial_t + \Lie{u_0})\rho_0w-w(\partial_t+\Lie{u_0})\rho_0 = (\partial_t + \Lie{u_0})\rho_0w-\rho_0w\frac{D_0\rho_0}{\rho_0}.
\label{eq:le1_NEW1}
\end{equation}
Applying the divergence to relation~\eqref{eq:le1_NEW1} and using identity~\eqref{eq:divLieuv} yields  
\begin{align}
\nabla\cdot(\rho_0\,\pert{}u) &=\nabla\cdot\left((\partial_t + \Lie{u_0})\rho_0w-\rho_0w\frac{D_0\rho_0}{\rho_0}\right)\nonumber\\ 
&= (\partial_t + \Lie{u_0})(\nabla\cdot(\rho_0 w)) - \Lie{\rho_0 w}\nabla\cdot u_0 -(\nabla\cdot(\rho_0 w))\frac{D_0\rho_0}{\rho_0}-(\rho_0w\cdot\nabla)\frac{D_0\rho_0}{\rho_0}\nonumber\\
&= D_0(\nabla\cdot(\rho_0 w)) -(\nabla\cdot(\rho_0 w))\frac{D_0\rho_0}{\rho_0} - (\rho_0w\cdot\nabla)\left(\frac{D_0\rho_0}{\rho_0}+\nabla\cdot u_0\right).
\label{eq:le1_NEW2}
\end{align}
The last term in expression~\eqref{eq:le1_NEW2} vanishes due to mass conservation~\eqref{eq:Eu02}, while, using that $-\rho_0^{-1}D_0\rho_0 = \rho_0 D_0(\rho_0^{-1})$, the first two terms may be combined to form the right-hand side of identity~\eqref{eq:divdEu}.   

We will now use identity~\eqref{eq:divdEu} to rewrite equation~\eqref{eq:ElE2}. 
To that end, we note that the second and third terms in equation~\eqref{eq:ElE2} combine to $\nabla\cdot(\rho_0\,\pert{}u)$, while the sum of the first and fourth terms can be rewritten as
\begin{align}
D_0\,\pert{}\rho + (\nabla\cdot u_0)\pert{}\rho &= D_0\left(\rho_0\frac{\pert{}\rho}{\rho_0}\right)+ \rho_0(\nabla\cdot u_0)\frac{\pert{}\rho}{\rho_0} = \rho_0D_0\left(\frac{\pert{}\rho}{\rho_0}\right) + \left(D_0\rho_0+\rho_0\nabla\cdot u_0\right)\frac{\pert{}\rho}{\rho_0}\nonumber\\
&= \rho_0D_0\left(\frac{\pert{}\rho}{\rho_0}\right),  
\label{eq:2}
\end{align}
where we in the last step have used mass conservation~\eqref{eq:Eu02}.
Thus, by identity~\eqref{eq:divdEu} equation~\eqref{eq:ElE2} is equivalent to 
\begin{equation}
\rho_0D_0\left(\frac{\pert{}\rho+\nabla\cdot(\rho_0w)}{\rho_0}\right) = 0.
\label{eq:LlE2Euler}
\end{equation}
\begin{rmrk}Note that since $\pert{}\rho + \nabla\cdot(\rho_0w) = \pert{}\rho + (w\cdot\nabla)\rho_0 + \rho_0\nabla\cdot w$ and by definition~\eqref{eq:dL}, we obtain from expression~\eqref{eq:LlE2Euler} that
\begin{equation}
\rho_0D_0\left(\frac{\pert{L}\rho}{\rho_0}+\nabla\cdot w\right) = 0,
\label{eq:LlE2}
\end{equation}  
which is the \emph{Lagrangian} linearization of~\eqref{eq:Eu2}. (Compare with equation~(58) in Gabard~\cite{Ga03}.)
\qed
\end{rmrk}
Tentatively assuming that equation~\eqref{eq:LlE2Euler} implies that
\begin{equation}
\pert{}\rho + \nabla\cdot(\rho_0w) = 0,
\label{eq:noResIndErho}
\end{equation}
and eliminating $\pert{}u,\pert{}\rho$, and $\pert{}p$ from equation~\eqref{eq:ElE1}, using relations~\eqref{eq:ElE3},~\eqref{eq:dEuRelw} and~\eqref{eq:noResIndErho}, we finally attain Galbrun's equation
\begin{equation}
\rho_0D_0(D_0w - (w\cdot\nabla)u_0) -\nabla(c_0^2\nabla\cdot(\rho_0w)) +\rho_0\left((D_0w - (w\cdot\nabla)u_0)\cdot\nabla\right)u_0+\frac{\nabla p_0}{\rho_0}\nabla\cdot(\rho_0w) = \rho_0\pert{}\varphi. 
\label{eq:G}
\end{equation}
Formulation~\eqref{eq:G} of Galbrun's equation has a different form compared to the standard formulation naturally obtained via Lagrangian linearization~\eqref{eq:GStandardForm},
\begin{equation}
\rho_0D_0^2w-\nabla(\rho_0c_0^2\nabla\cdot w)+(\nabla p_0)\nabla\cdot w - (\nabla w)^T\nabla p_0 -\rho_0(w\cdot\nabla)\varphi_0 = \rho_0\,\pert{}\varphi.
\tag{v.s.~\eqref{eq:GStandardForm}}
\end{equation}
Nevertheless, a lengthy direct calculation yields that formulations~\eqref{eq:G} and~\eqref{eq:GStandardForm} are equivalent. 

The above presentation deliberately exposes a potential weak link in the derivation of Galbrun's equation, namely the transition from equations~\eqref{eq:LlE2Euler} to~\eqref{eq:noResIndErho}.  
This transition is often referred to as the ``no resonance'' assumption in the literature---a terminology introduced by Godin~\cite{Go97}---and it will be further analyzed in the sequel.

\section{Preliminaries}
\label{sec:preliminaries}
Before continuing the investigations into the well-posedness of Galbrun's equation, we briefly recall an abstract framework for time dependent Friedrichs' systems that will be extensively used to assess well-posedness of various initial--boundary value problems that are closely related to Galbrun's equation.  
The time dependent framework that we rely on was recently presented by Burazin and Erceg~\cite{BuEr16} and constitutes an extension of a framework for time independent Friedrich's systems presented by Ern et al.~\cite{ErGuCa07}.

We use the notation from the article by Ern et al.~\cite{ErGuCa07}. 
Let $L$ be a Hilbert space with inner product $(\cdot,\cdot)_L$ and norm $\|\cdot\|_L$,  and  $\mathcal{D}$ a dense subspace of $L$. 
Moreover, $L$ is identified with its dual $L'$.
Let $T:\mathcal{D}\to L$ and $\tilde{T}:\mathcal{D}\to L$ be two linear operators that satisfy
\begin{align}
&(T\phi,\psi)_L = (\phi,\tilde{T}\psi)_L~\text{for all $\phi,\psi\in \mathcal{D}$,}\label{eq:T1}\tag{T1}\\
&\|(T+\tilde{T})\phi\|_L \lesssim\|\phi\|_L~\text{for all $\phi\in\mathcal{D}$.}\label{eq:T2}\tag{T2}
\end{align}
By $W_0$, we denote the completion of $\mathcal{D}$ in the inner product $(\cdot,\cdot)_L + (T\cdot,T\cdot)_L$ (or equivalently in the inner product $(\cdot,\cdot)_L + (\tilde{T}\cdot,\tilde{T}\cdot)_L$). 
As detailed by Antoni{\'c} and Burazin~\cite{AnBu10}, the operators $T$ and $\tilde{T}$ can be extended, first by density and then by adjoints, to bounded operators from $L$ to $W_0'$.
Abusing the notation, we still denote these extensions $T,\tilde{T}\in\mathcal{L}(L;W_0')$.
The graph space $W = \{\xi\in L\mid T\xi\in L\}~(~= \{\xi\in L\mid \tilde{T}\xi\in L\})$ is a Hilbert space when equipped with the graph inner product $(\cdot,\cdot)_W = (\cdot,\cdot)_L + (T\cdot,T\cdot)_L$.

The boundary operator $\mathfrak{D}\in\mathcal{L}(W;W')$ is defined for $\xi,\tilde{\xi}\in W$ through
\begin{equation}
\langle \mathfrak{D}\xi,\tilde{\xi} \rangle_W = (T\xi,\tilde{\xi})_L - (\xi,\tilde{T}\tilde{\xi})_L = \langle \mathfrak{D}\tilde{\xi},\xi\rangle_W.
\label{eq:D}
\end{equation}
Let $V$ and $\tilde{V}$ be subspaces of $W$ that satisfy the conditions
\begin{align}
\langle \mathfrak{D}\xi,\xi\rangle_W\geq 0~\text{for all }\xi\in V~&\text{and }\langle \mathfrak{D}\tilde{\xi},\tilde{\xi}\rangle_W\leq 0~\text{for all }\tilde{\xi}\in \tilde{V},
\label{eq:V1}\tag{V1}\\
V = \mathfrak{D}(\tilde{V})^0~&\text{and }\tilde{V} = \mathfrak{D}(V)^0\label{eq:V2}\tag{V2},
\end{align} 
where $\mathfrak{D}(V)^0 = \{\tilde{\xi} \in W\mid \langle \mathfrak{D}\xi,\tilde\xi\rangle_W = 0\,\forall\xi\in V\}$ denotes the annihilator of $\mathfrak{D}(V)$, and $\mathfrak{D}(\tilde{V})^0$ denotes the annihilator of $\mathfrak{D}(\tilde{V})$.

The abstract Cauchy problem related to the operator $T$ is given by
\begin{subequations}
	\begin{align}
		(\partial_t + T)\xi &= f~\text{for }t \in (0,\tau),\label{eq:abstractCauchyEq}\\
		\xi &= \xi_\I~\text{at }t = 0\label{eq:abstractCauchyIC},
	\end{align}
	\label{eq:abstractCauchy}
\end{subequations}
where $\xi:[0,\tau)\to L$ for some $\tau\in(0,\infty)$, $f:(0,\tau)\to L$ and $\xi_\I\in L$.
We define $\lambda_0$ by
\begin{equation}
2\lambda_0 = \max\left\{0,-\inf_{\phi\in \mathcal{D}\setminus \{0\}}\frac{((T+\tilde{T})\phi,\phi)_L}{\|\phi\|_L^2}\right\}\leq \sup_{\phi\in \mathcal{D}\setminus \{0\}}\frac{\|(T+\tilde{T})\phi\|_L}{\|\phi\|_L} <\infty,
\label{eq:lambda0}
\end{equation}
where the last inequality follows from assumption~\eqref{eq:T2}.
Then, the operator $\mathcal{A}_{\lambda_0}: V\subset L\to L$, $\mathcal{A}_{\lambda_0} = -(T + \lambda_0 I)\rvert_{V}$ is the infinitesimal generator of a contraction $C_0$-semigrouop $(S_{\lambda_0}(t))_{t\geq 0}$ \cite{BuEr16}[Theorem 2], from which it follows that the abstract Cauchy problem~\eqref{eq:abstractCauchy} is uniquely solvable~\cite{BuEr16}[Corollary 1].
\begin{thm}
\label{th:gen}
Let $T$ and $\tilde{T}$ satisfy conditions~\eqref{eq:T1} and~\eqref{eq:T2}, let $V\subset W$ and $\tilde{V}\subset W$ satisfy conditions~\eqref{eq:V1} and~\eqref{eq:V2}, and let $f\in L^1((0,\tau);L)$.
Then, for every $\xi_\I\in L$, problem~\eqref{eq:abstractCauchy} has a unique mild solution $\xi\in C([0,\tau];L)$ given by 
\begin{equation}
\xi(t) = e^{\lambda_0 t}S_{\lambda_0}(t)\xi_\I + \int\limits_{0}^{t}e^{{\lambda_0} (t-s)}S_{\lambda_0}(t-s)f(s)\,\mathrm{d}s,
\label{eq:mildSolution}
\end{equation} 
where $\lambda_0$ is given by expression~\eqref{eq:lambda0}.  
\end{thm}

Finally, we note that formula~\eqref{eq:mildSolution} yields the estimate
\begin{equation}
\|\xi(t)\|_L \leq e^{\lambda_0\tau}\left(\|\xi_\I\|_L + \int\limits_0^t{\|f(s)\|_L\,\mathrm{d}s} \right)\!\!, 
\label{eq:enEst}
\end{equation}  
which shows that the solution depends continuously on data.

\subsection{Notations}
By $(\cdot,\cdot)$ we denote the standard $L^2(\Omega)^k$ inner product, that is,
\begin{equation}
(u,v) := \int\limits_{\Omega} {u^T\!v} = \int\limits_{\Omega} \sum_{i = 1}^k{u_i v_i}.
\label{eq:L2Innerproduct}
\end{equation}
Analogously, $(\cdot,\cdot)_{\partial\Omega}$ denotes the standard $L^2(\partial\Omega)^k$ inner product, that is,
\begin{equation}
(u,v)_{\partial\Omega} := \int\limits_{\partial\Omega} {u^T\!v} = \int\limits_{\partial\Omega} \sum_{i = 1}^k{u_i v_i}.
\label{eq:L2BndInnerproduct}
\end{equation}

Let $A,B$ be normed spaces with norms $\|\cdot\|_A,\|\cdot\|_B$. We write $\|u_A\|_A \lesssim \|u_B\|_B$ if there exists a $C>0$ independent of $u_B$ such that $\|u_A\|_A \leq C \|u_B\|_B$. 

\section{Existence of solutions to Galbrun's equation}
\label{sec:existence} 
In this section, we present a scheme to generate solutions to Galbrun's equation~\eqref{eq:G} from solutions to linearized Euler's equations~\eqref{eq:ElE}.
The idea is that if $\pert{}u$ and $\pert{}\rho$ are solutions to linearized Euler's equations~\eqref{eq:ElE}, then the Lagrangian displacement $w$ may be found by solving equation~\eqref{eq:dEuRelw}. Moreover, provided that the ``no resonance'' assumption is satisfied, then this $w$ is a solution to Galbrun's equation~\eqref{eq:G}.
To validate such a scheme, we investigate
\begin{itemize}
	\item existence of solutions to linearized Euler's equations~\eqref{eq:ElE},
	\item existence of solutions to equation~\eqref{eq:dEuRelw}, and
	\item conditions that guarantee fulfillment of the ``no resonance'' assumption.   
\end{itemize}   

\subsection{Dissecting the ``no resonance'' assumption}
\label{sec:noResAss}
If we introduce the quantity 
\begin{equation}
h := \rho_0^{-1}(\pert{}\rho + \nabla\cdot(\rho_0 w)),
\label{eq:h}
\end{equation} 
the ``no resonance'' assumption takes the form
\begin{equation}
\rho_0D_0 h = 0 \implies h = 0.
\label{eq:noResInh}
\end{equation} 
Anticipating that the initial value problem (or depending on the situation, the initial--boundary value problem) that corresponds to the equation to the left of the implication~\eqref{eq:noResInh} is well-posed, the desired implication would follow if we could provide vanishing data for $h$.
While pursuing this idea, we will reveal that the ``no resonance'' assumption in some cases imposes a restriction on the Lagrangian displacement $w$ only, and not on the Eulerian perturbations $\pert{}\rho$ and $\pert{}u$.

Assume that $\Omega\subset\mathbb{R}^d$ is open, bounded, connected, and lies locally on one side of its Lipschitz boundary $\partial\Omega$.
By $n$ we denote the outward unit normal field on $\partial\Omega$.
We partition the boundary into three disjoint parts depending on the sign of $n\cdot u_0$
	\begin{equation}
		\Gamma_- = \{x\in\partial\Omega\mid n\cdot u_0 < 0\},
		\Gamma_0 = \{x\in\partial\Omega\mid n\cdot u_0 = 0\},~\text{and }
		\Gamma_+ = \{x\in\partial\Omega\mid n\cdot u_0 > 0\},	\label{eq:partdOmega}	
	\end{equation}
and we assume that $u_0$ is such that this partition does not vary with time and $\mathrm{dist}(\Gamma_-,\Gamma_+)>0$.

Unless otherwise stated, we assume that the background flow is steady and that the background flow quantities are Lipschitz continuous in $\bar{\Omega}$. Thus the background flow quantities are bounded in $\bar{\Omega}$ and have bounded first order spatial derivatives almost everywhere in $\bar{\Omega}$.
Moreover, we assume that the density and the speed of sound are bounded away from zero, that is,  
\begin{equation}
\begin{aligned}
\inf\limits_{\bar{\Omega}}\rho_0 & =:  \underline{\rho_0} > 0,\\
\inf\limits_{\bar{\Omega}} c_0 & =: \underline{c_0} > 0.
\end{aligned}
\label{eq:rho0Andc0Positive}
\end{equation}

The initial--boundary value problem related to the ``no resonance'' assumption reads
\begin{subequations}
\begin{align}   
\rho_0D_0 h &= 0&~\text{in $\Omega$ for all $t\in(0,\tau)$,}\label{eq:noResEq}\\
h &= h_\I&~\text{in $\Omega$ at $t = 0$,}\label{eq:noResIC}\\
h &= h_-&~\text{on $\Gamma_-$ for all $t\in(0,\tau)$},
\label{eq:noResBC}
\end{align}
\label{eq:noResIVBP}%
\end{subequations}
where $0<\tau<\infty$.
\begin{thm} 
\label{th:noResWellPo} 
Assume that the background flow is steady.
If $h_\I\in L^2(\Omega)$ and $h_- =0$, then the initial--boundary value problem~\eqref{eq:noResIVBP} is well-posed.
\end{thm}
\begin{proof}
With notation as in Section~\ref{sec:preliminaries}, we let $L = L_{\rho_0}^2(\Omega)$ with inner product $(\cdot,\cdot)_L = (\rho_0\cdot,\cdot)$.
By assumption, $0<\inf_{\bar{\Omega}} \rho_0 \leq \sup_{\bar{\Omega}} \rho_0 < \infty$, which implies that $L$ is topologically equivalent to $L^2(\Omega)$. 
Therefore, in this case, we let ${C}_0^\infty(\Omega)$ serve as the dense set $\mathcal{D}\subset L$ described in Section~\ref{sec:preliminaries}. 
In equation~\eqref{eq:noResEq}, we find $T = u_0\cdot\nabla$. 
By mass conservation~\eqref{eq:Eu02} it holds that $\nabla\cdot(\rho_0u_0) = 0$, which implies that the formal adjoint of $T$ in $L$ is $\tilde{T} = -u_0\cdot\nabla = -T$.
Thus, $T$ and $\tilde{T}$ satisfy conditions~\eqref{eq:T1} and~\eqref{eq:T2}, and definition~\eqref{eq:lambda0} yields $\lambda_0 = 0$, since $T+\tilde{T} \equiv 0$. 
Due to the assumption that $\mathrm{dist}(\Gamma_-,\Gamma_+)>0$, there is a continuous trace operator $\gamma:W\to L^2_{|n\cdot\rho_0u_0|}(\partial\Omega)$, where $W = \{h\in L\mid u_0\cdot \nabla h \in L\}$~\cite{ErGu06a}[Lemma 3.1]. 
Let $V = \{h\in W\mid (\gamma h)\rvert_{\Gamma_-} = 0\}$ and $\tilde{V} = \{h\in W\mid (\gamma h)\rvert_{\Gamma_+} = 0\}$, then conditions~\eqref{eq:V1} and~\eqref{eq:V2} are satisfied~\cite{ErGuCa07}[Lemma 5.2].   
Hence, Theorem~\ref{th:gen} and estimate~\eqref{eq:enEst} apply.  
\end{proof} 

\begin{rmrk}Note that a similar estimate to estimate~\eqref{eq:enEst} can be derived directly for problem~\eqref{eq:noResIVBP}, even for unsteady background flows as long as the partitioning~\eqref{eq:partdOmega} does not vary with time.
Formally, multiplying equation~\eqref{eq:noResEq} by $h$,  integrating over $\Omega$, using integration-by-parts formula~\eqref{eq:ip3} and invoking boundary condition~\eqref{eq:noResBC}, we find that
\begin{equation}
\frac{d}{dt}(\rho_0h, h) = \int\limits_{\Gamma_-}\rho_0|n\cdot u_0|h^2 - \int\limits_{\Gamma_+}\rho_0|n\cdot u_0|h^2\leq \int\limits_{\Gamma_-}\rho_0|n\cdot u_0|h_-^2.
\label{eq:noResPreEst}
\end{equation}
Integrating inequality~\eqref{eq:noResPreEst} over the time interval $(0,t)$ and invoking initial condition~\eqref{eq:noResIC}, we obtain the estimate
\begin{equation}
(\rho_0 h(t), h(t))  \leq (\rho_0 h_\I, h_\I) + \int\limits_0^t\int\limits_{\Gamma_-}\rho_0|n\cdot u_0|h_-^2. 
\label{eq:noResEst}
\end{equation}
\qed
\end{rmrk}

In either case, Theorem~\ref{th:noResWellPo} or estimate~\eqref{eq:noResEst} show that if $h_\I = h_- = 0$, then the ``no resonance'' assumption~\eqref{eq:noResInh} follows. 
The issue is whether vanishing data for $h$ in problem~\eqref{eq:noResIVBP} can be provided for $h$ defined as in expression~\eqref{eq:h}.
To that end, assume that $\pert{}u$, $\pert{}\rho$ satisfy the linearized Euler's equations~\eqref{eq:ElE} with suitable initial and boundary conditions, and that the initial datum $\pert{}\rho_\I$ of $\pert{}\rho$ belongs to $L^2(\Omega)$.
We start by investigating the initial condition
\begin{equation}
0= h\rvert_{t = 0} = (\pert{}\rho + \nabla\cdot(\rho_0 w))\rvert_{t = 0} = \pert{}\rho\rvert_{t = 0} + \nabla\cdot(\rho_0 w\rvert_{t = 0})~\text{in }\Omega.
\label{eq:noResHomIC}
\end{equation}
Equation~\eqref{eq:noResHomIC} requires that the initial datum $w_\I$ of the Lagrangian displacement satisfies $\nabla\cdot(\rho_0 w_\I) = -\pert{}\rho_\I$ in $\Omega$, which can be achieved by defining $w_\I = \rho_0^{-1}\nabla v_\I$, where $v_\I\in H^1_0(\Omega)$ is the solution to $-\Delta v_\I = \pert{}\rho_\I$ in $\Omega$.
Thus, using $w\rvert_{t = 0} = w_\I$ as the initial condition for the Lagrangian displacement, we obtain that $h\rvert_{t = 0} = 0$.
For the case when the background flow is everywhere tangential to $\partial\Omega$---$\partial\Omega = \Gamma_0$ and $\Gamma_- = \Gamma_+ = \emptyset$---no boundary condition is needed, and the ``no resonance'' assumption holds if and only if $h\rvert_{t = 0} = 0$, which can be achieved by adjusting the initial datum of $w$ as demonstrated above.
We note that, in this particular case, the ``no resonance'' assumption imposes no restriction on the initial datum $\pert{}\rho_\I\in L^2(\Omega)$ of $\pert{}\rho$.
However, when the background flow is \emph{not} everywhere tangential to $\partial\Omega$, imposing homogeneous data for $h$ at the boundary part $\Gamma_-$ appears unfortunately to be difficult.
Indeed, we would like to impose
\begin{equation}
0 = h\rvert_{\Gamma_-} = (\pert{}\rho + \nabla\cdot(\rho_0 w))\rvert_{\Gamma_-} = \pert{}\rho\rvert_{\Gamma_-} + (\nabla\cdot(\rho_0 w))\rvert_{\Gamma_-}.
\label{eq:noResHomBC}
\end{equation}
Contrary to condition~\eqref{eq:noResHomIC} that directly translates into a condition on the initial datum of $w$, it is not possible to convert expression~\eqref{eq:noResHomBC} into a condition on the boundary datum of $w$ on $\Gamma_-$.  

\subsection{Well-posedness of linearized Euler's equations}
\label{sec:ElE_WP}
In this section we prove well-posedness of an initial--boundary value problem for linearized Euler's equations for steady background flows on a bounded domain by employing the framework for abstract Friedrichs' systems that was briefly recalled in Section~\ref{sec:preliminaries}.
Following Kreiss and Lorenz~\cite{KrLo89}[Chapter~8.3], we introduce the scaled quantity 
\begin{equation}
\pert{}\hat{\rho} = c_0\frac{\pert{}\rho}{\rho_0}.
\label{eq:dHatRho}
\end{equation}
Then the linearized Euler's equations~\eqref{eq:ElE} may be rewritten in the form
\begingroup
\renewcommand*{\arraystretch}{1.8}
\begin{equation}
\left(\rho_0\partial_t +
\begin{pmatrix}
	\rho_0u_0\cdot\nabla & \nabla(\rho_0c_0 \cdot)\\
	\rho_0c_0\nabla\cdot & \rho_0u_0\cdot\nabla
\end{pmatrix}
+\rho_0c_0
\begin{pmatrix}
	\dfrac{\nabla u_0}{c_0} & -\dfrac{\nabla\rho_0}{\rho_0}\\
	\dfrac{\nabla\rho_0}{\rho_0}\cdot & -\dfrac{D_0c_0}{c_0^2}
\end{pmatrix}
\right)
\begin{pmatrix}
	\pert{}u\\
	\pert{}\hat{\rho}
\end{pmatrix}
= \rho_0
\begin{pmatrix}
	\pert{}\varphi\\
	0
\end{pmatrix}.
\label{eq:scaledElE}
\end{equation}
\endgroup

Similarly as in Section~\ref{sec:noResAss}, we assume that $\Omega\subset\mathbb{R}^d$ is open, bounded, connected, and lies locally on one side of its boundary $\partial\Omega$, that the background flow is steady and that the background flow quantities are Lipschitz continuous in $\bar{\Omega}$, and that the density and the speed of sound are bounded away from zero~\eqref{eq:rho0Andc0Positive}. 
In addition we assume that $\partial\Omega$ is $C^1$-regular with a Lipschitz continuous unit normal vector field $n$, and that the background flow is everywhere tangential to $\partial\Omega$.   
Recall from Section~\ref{sec:noResAss} that, since $n\cdot u_0 = 0$ on $\partial\Omega$, the ``no resonance'' assumption can be enforced by appropriately choosing the initial datum for the Lagrangian displacement.

To form an initial--boundary value problem, we supply to equation~\eqref{eq:scaledElE} the initial and boundary conditions 
\begin{align}
(\pert{}u,\pert{}\hat{\rho}) &= (\pert{}u_\I,\pert{}\hat{\rho}_\I)&\text{in $\Omega$ at $t = 0$,}\\
-n\cdot\pert{}u + Y\pert{}\hat{\rho} &= 0&\text{on $\partial\Omega$ for all $t\in(0,\tau)$,}
\label{eq:lEuBC}
\end{align}
where $Y:\partial\Omega \to [0,\infty)$ is a Lipschitz continuous (dimensionless) admittance function, and $0<\tau<\infty$.
The spatially variable admittance function $Y$ allows for the interpolation between the boundary conditions $n\cdot \pert{}u = 0$ and $\pert{}\hat{\rho} -n\cdot\pert{}u = 0$.  
The former holds at an impenetrable wall and the latter can be used as an artificial boundary condition to truncate an unbounded domain. 
Albeit not exactly representable, the boundary condition $\pert{}\hat{\rho} = 0$, which would correspond to $Y=\infty$, can be approximately enforced by choosing $Y$ to be large.
We introduce $\xi = (\xi_1,\xi_2) = (\pert{}u,\pert{}\hat{\rho})$, $\xi_\I = (\xi_{1,\I},\xi_{2,\I}) = (\pert{}u_\I,\pert{}\hat{\rho}_\I)$, $f = (\pert{}\varphi,0)$, and the block operators $T,A$, and $B$ defined by
\begin{equation}
\begingroup
\renewcommand*{\arraystretch}{1.8}
\rho_0T = A + B = 
\begin{pmatrix}
	\rho_0u_0\cdot\nabla & \nabla(\rho_0c_0 \cdot)\\
	\rho_0c_0\nabla\cdot & \rho_0u_0\cdot\nabla
\end{pmatrix}
+\rho_0c_0
\begin{pmatrix}
	\dfrac{\nabla u_0}{c_0} & -\dfrac{\nabla\rho_0}{\rho_0}\\
	\dfrac{\nabla\rho_0}{\rho_0}\cdot & -\dfrac{D_0c_0}{c_0^2}
\end{pmatrix}.
\label{eq:T}
\endgroup
\end{equation}
Moreover, introducing the Lipschitz continuous unit vector field $e=(-n,Y)/|(-n,Y)|$ on the boundary allows us to compactly state initial--boundary value problem~\eqref{eq:scaledElE}--\eqref{eq:lEuBC} as
\begin{subequations}
	\begin{align}
		(\partial_t + T)\xi &= f~&\text{in $\Omega$ for all $t\in(0,\tau)$,}\label{eq:IVP_ElEEQ}\\
		\xi &= \xi_\I~&\text{ in $\Omega$ at $t = 0,$}\label{eq:IVP_ElEIC}\\
		\sqrt{1+Y^2}e^T\xi &= 0~&\text{on $\partial\Omega$ for all $t\in(0,\tau)$.}\label{eq:IVP_ElEBC}
	\end{align}
	\label{eq:IVP_ElE}%
\end{subequations}

In principle, nonzero boundary data can also be handled in problem~\eqref{eq:IVP_ElE}.
Formally, a problem in $\xi$ with \emph{nonzero} boundary datum $g$ transforms to a problem in $(\xi-\tilde{g})$ with \emph{zero} boundary datum by the splitting $\xi = (\xi-\tilde{g}) + \tilde{g}$, where $\tilde{g}$ is a kind of extension of $g$ to $\Omega\times (0,\tau)$ that satisfies $\sqrt{1+Y^2}e^T\tilde{g} = g$ on $\partial\Omega$.

With notation as in Section~\ref{sec:preliminaries}, we let $L = L_{\rho_0}^2(\Omega)^{d+1}$ with inner product $(\cdot,\cdot)_{L} = (\rho_0 \cdot, \cdot)$. 
Since  $0<\inf_{\bar{\Omega}}\rho_0 \leq \sup_{\bar{\Omega}}\rho_0 < \infty$, $L$ is topologically equivalent to $L^2(\Omega)^{d+1}$, and $\mathcal{D} = {C}_{0}^{\infty}(\Omega)^{d+1}$ is dense in $L$. 
Let $T$ be defined as in relation~\eqref{eq:T} and let $\rho_0\tilde{T} = -A + B^T$. Then conditions~\eqref{eq:T1} and~\eqref{eq:T2} in Section~\ref{sec:preliminaries} are satisfied; that is, $\tilde{T}$ is the formal adjoint of $T$ in $L$, and $T + \tilde{T} = \rho_0^{-1}(B + B^T)$ is a bounded operator on $L$. 
We define the graph space $W = \{\xi\in L\mid T\xi \in L\} = \{\xi\in L\mid \tilde{T}\xi \in L\}$, which is a Hilbert space in the graph norm
\begin{align}
\|\xi\|_W^2 &= \tau_0^{-2}\|\rho_0\xi\|^2 + \|A\xi\|^2\nonumber\\  &=\tau_0^{-2}\|\rho_0\xi\|^2 + \|(\rho_0u_0\cdot\nabla)\xi_1 + \nabla(\rho_0c_0\xi_2)\|^2 + \|\rho_0c_0\nabla\cdot\xi_1 + \rho_0u_0\cdot{\nabla}\xi_2\|^2,
\label{eq:normW}
\end{align}
where the time scale $\tau_0 > 0$ has been included to make the terms dimensionally consistent.   
\begin{rmrk}
Unless $\xi$ is regular enough, the individual terms in the last two norms might not be well-defined---analogously as the $\partial_1 u_1$ term of $\nabla\cdot u = \partial_1 u_1 + \partial_2 u_2+ \partial_3 u_3$ might not be well-defined for $u\in H_{\mathrm{div}}(\Omega)$.
Moreover, since we have assumed that $0<\inf_{\bar{\Omega}} \rho_0\leq \sup_{\bar{\Omega}} \rho_0 < \infty$, and since $B$ is a bounded operator on $L^2(\Omega)^{d+1}$, $\|\cdot\|_W$ in expression~\eqref{eq:normW} is equivalent to the ``standard'' graph norm $\sqrt{\|\tau_0^{-1}\cdot\|_L^2 + \|T\cdot\|_L^2}$.
\qed
\end{rmrk}

The boundary operator $\mathfrak{D}: W\to W'$ is given by
\begin{equation}
\langle \mathfrak{D}\xi,\tilde{\xi} \rangle_W := (T\xi, \tilde{\xi})_L - (\xi,\tilde{T}\tilde{\xi})_L = (A\xi,\tilde{\xi}) + (\xi,A\tilde{\xi}),
\label{eq:DlE}
\end{equation}
and we note that for $\phi,\psi \in C^{1}(\bar{\Omega})^{d+1}$, it has the representation
\begin{equation}
\left\langle \mathfrak{D}\phi,\psi \right\rangle_W = (A\phi,\psi) + (\phi,A\psi) =\int\limits_{\partial\Omega}\psi^T A(n)\phi = \int\limits_{\partial\Omega}\rho_0c_0(n\cdot \phi_1\psi_2 + n\cdot\psi_1\phi_2),
\label{eq:DSmoothFunctions}
\end{equation}
where 
\begin{equation}
A(n) = \rho_0c_0\begin{pmatrix} 0_{3\times 3} & n\\ n\cdot & 0 \end{pmatrix},
\label{eq:An}
\end{equation}
since $n\cdot u_0 = 0$ on $\partial\Omega$.

We will now proceed to define a trace operator for functions in $W$ that provides a sound mathematical treatment of boundary condition~\eqref{eq:IVP_ElEBC}.
Our approach is based on the work of Rauch~\cite{Ra85}, who investigates initial--boundary value problems with \emph{vector valued} boundary conditions that are characterized using quotient spaces on the boundary. 
Here, we present a more elementary, nevertheless equivalent, characterization for the \emph{scalar valued} boundary condition~\eqref{eq:IVP_ElEBC}.    

For $\xi\in C^{1}(\bar{\Omega})^{d+1}$, we define a linear operation $\gamma_e \xi = e^T\xi\rvert_{\partial\Omega}$.
Observe that $e^T = b^TA(n)$, where $A(n)$ is defined by expression~\eqref{eq:An} and $b$ is a Lipschitz continuous vector field on $\partial\Omega$ ($b = (Yn,-1)/(\rho_0c_0\sqrt{1+Y^2})$).
For any $v\in H^{1/2}(\partial\Omega)$, we obtain by integration-by-parts formula~\eqref{eq:DSmoothFunctions} that
\begin{equation}
\int\limits_{\partial\Omega}(\gamma_e\xi)v =\int\limits_{\partial\Omega}(e^T\xi)v = \int\limits_{\partial\Omega}(b^TA(n)\xi)v = \int\limits_{\partial\Omega}(vb)^TA(n)\xi = \int\limits_{\Omega}(\gamma_0^\ast(vb))^T A\xi + \int\limits_{\Omega}(A\gamma_0^\ast(vb))^T \xi,
\label{eq:tr2}
\end{equation}
where $\gamma_0^*: H^{1/2}(\partial\Omega)^{d+1}\to H^1(\Omega)^{d+1}$ denotes a (bounded linear) right inverse of the standard trace operator $\gamma_0:H^{1}(\Omega)^{d+1}\to H^{1/2}(\partial\Omega)^{d+1}$, and where we have used that $vb\in H^{1/2}(\partial\Omega)^{d+1}$ since $b$ is Lipschitz continuous.
From relation~\eqref{eq:tr2} we obtain the estimate
\begin{equation}
\begin{aligned}
\left|\,\int\limits_{\partial\Omega}(\gamma_e\xi)v\right| &\lesssim \|\gamma_0^\ast(vb)\|_{H^1(\Omega)^{d+1}}\|\xi\|_W\lesssim \|vb\|_{H^{1/2}(\partial\Omega)^{d+1}}\|\xi\|_W\lesssim \|v\|_{H^{1/2}(\partial\Omega)}\|\xi\|_W,
\end{aligned}
\label{eq:tr3}
\end{equation} 
from which it follows that $\|\gamma_e \xi\|_{H^{-1/2}(\partial\Omega)}\lesssim \|\xi\|_W$.
Since $C^1(\bar{\Omega})^{d+1}$ is dense in $W$~\cite{Ra85}[Proposition 1], the operator $\gamma_e$ extends uniquely to a bounded linear operator $W \to H^{-1/2}(\partial\Omega)$, which we still denote by $\gamma_e$.
We say that $\xi\in W$ satisfies boundary condition~\eqref{eq:IVP_ElEBC} if and only if $\gamma_e \xi = 0$ and define $V = \mathrm{ker}\,\gamma_e\subset W$.

To analyze initial--boundary value problem~\eqref{eq:IVP_ElE}, we also need to consider functions that satisfy the adjoint boundary condition
\begin{equation}
\sqrt{1+Y^2}\tilde{e}^T\xi = 0~\text{on $\partial\Omega$ for all $t\in(0,\tau)$},\label{eq:adjointBC}
\end{equation}
where $\tilde{e} = (n,Y)/|(n,Y)|$.
Since $\tilde{e}^T = \tilde{b}^TA(n)$, where $\tilde{b}$ is a Lipschitz continuous vector field on $\partial\Omega$ ($\tilde{b} = (Yn,1)/(\rho_0c_0\sqrt{1+Y^2})$), we
may proceed analogously as in the case of $e$ to define the trace operator $\gamma_{\tilde{e}}: W\to H^{-1/2}(\partial\Omega)$ and the subspace $\tilde{V} = \mathrm{ker}\, \gamma_{\tilde{e}}$.
The following lemma expresses the fundamental geometric relation between boundary conditions~\eqref{eq:IVP_ElEBC} and~\eqref{eq:adjointBC}.
\begin{lem}
\label{lem:basicAnOrth}
For each $x\in \partial\Omega$ define $N(x) = \{\xi\in \mathbb{R}^{d+1}\mid e(x)^T\xi = 0\}$ and $\tilde{N}(x) = \{\tilde{\xi}\in \mathbb{R}^{d+1}\mid \tilde{e}(x)^T\tilde{\xi} = 0\}$, then 
\begin{align}
\tilde{N}(x) &= (A(n(x))N(x))^\perp = \{\tilde{\xi}\in\mathbb{R}^{d+1}\mid \tilde{\xi}^T A(n(x))\xi = 0~\forall \xi\in N(x)\},\\ 
N(x) &= (A(n(x))\tilde{N}(x))^\perp = \{\xi\in\mathbb{R}^{d+1}\mid \xi^T A(n(x))\tilde{\xi} = 0~\forall \tilde{\xi}\in \tilde{N}(x)\},\\
\mathrm{ker}\,A(n(x))&\subset N(x)\cap \tilde{N}(x).\label{eq:kerAn}
\end{align}
\begin{proof}
From definition~\eqref{eq:An} it follows that
\begin{equation}
\tilde{\xi}^{\,T}A(n(x))\xi = \xi^{T}A(n(x))\tilde{\xi} = \rho_0c_0\left(\tilde{\xi_2}n(x)\cdot\xi_1+\xi_2n(x)\cdot\tilde{\xi_1}\right).
\label{eq:xiTildeAnXi}
\end{equation}
On the one hand, if $\xi\in N(x)$, then $0 = \sqrt{1+Y(x)^2}e(x)^T\xi = -n(x)\cdot \xi_1 + Y(x)\xi_2$, which implies that $n(x)\cdot\xi_1 = Y(x)\xi_2$.
On the other hand, if $\tilde{\xi}\in\tilde{N}(x)$, then $0 = \sqrt{1+Y(x)^2}\tilde{e}(x)^T\xi = n(x)\cdot \tilde{\xi}_1 + Y(x)\tilde{\xi}_2$, which implies that $n(x)\cdot\tilde{\xi}_1 = -Y(x)\tilde{\xi}_2$. It follows from expression~\eqref{eq:xiTildeAnXi} that $\tilde{\xi}^{\,T}A(n(x))\xi = \xi^{T}A(n(x))\tilde{\xi} = 0$ for any pair $\xi\in N(x)$ and $\tilde{\xi}\in \tilde{N}(x)$. Thus, $\tilde{N}(x) \subset (A(n(x))N(x))^\perp$ and $N(x) \subset (A(n(x))\tilde{N}(x))^\perp$. 
We now proceed to demonstrate the reverse inclusions.
To that end, let us first assume that $\tilde{\xi}\in\mathbb{R}^{d+1}$ is such that $\tilde{\xi}^TA(n(x))\xi = 0$ for any $\xi\in N(x)$.  
Using that $n(x)\cdot\xi_1 = Y(x)\xi_2$ in expression~\eqref{eq:xiTildeAnXi}, we find that
\begin{equation}
0 = \rho_0 c_0\,\left(n(x)\cdot\tilde{\xi}_1 + Y(x)\tilde{\xi}_2\right)\xi_2~\text{for all }\xi_2\in\mathbb{R},
\label{eq:orth1}
\end{equation}
from which it follows that $\tilde{e}(x)^T\tilde{\xi} = 0$, that is, $\tilde{\xi}\in \tilde{N}(x)$. 
We have thus demonstrated that $(A(n(x))N(x))^\perp \subset \tilde{N}(x)$ and note that the missing inclusion $(A(n(x))\tilde{N}(x))^\perp \subset N(x)$ can be demonstrated analogously. 
Finally, definition~\eqref{eq:An} implies that $\mathrm{ker}\,A(n(x)) = \{\xi\in\mathbb{R}^{d+1}\mid n(x)\cdot\xi_1 = 0~\text{and }\xi_2 = 0\}$, and thus inclusion~\eqref{eq:kerAn} holds.   
\end{proof}
\end{lem}  
\begin{rmrk}
Note the symmetry in the relationship between $N(x)$ and $\tilde{N}(x)$.
\qed 
\end{rmrk}   

The following theorem, which is due to Rauch~\cite{Ra85}[Theorem 4], establishes denseness of $C^1(\bar{\Omega})^{d+1}\cap V$ in $V$ and, by symmetry, denseness of $C^1(\bar{\Omega})^{d+1}\cap \tilde{V}$ in $\tilde{V}$.
\begin{thm}
\label{th:Ra85th4}
If the boundary is a characteristic surface of constant multiplicity, that is, the dimension of $\mathrm{ker} A(n)$ does not vary along the boundary, then $C^1(\bar{\Omega})^{d+1}\cap V$ is dense in $V.$
\end{thm}

\begin{rmrk}
Note that for $A(n)$ defined in expression~\eqref{eq:An}, $\mathrm{dim}\,\mathrm{ker} A(n) = d-1$, which implies that Theorem~\ref{th:Ra85th4} is indeed applicable.
Moreover, note that functions in $C^1(\bar{\Omega})^{d+1}\cap V$ satisfy boundary condition~\eqref{eq:IVP_ElEBC}  pointwise, while functions in $C^1(\bar{\Omega})^{d+1}\cap \tilde{V}$ satisfy adjoint boundary condition~\eqref{eq:adjointBC} pointwise.     
\qed
\end{rmrk}       

Using Theorem~\ref{th:Ra85th4}, we now demonstrate that boundary operator~\eqref{eq:DlE} is non-negative on $V$ and non-positive on $\tilde{V}$. 
For $\xi\in C^1(\bar{\Omega})^{d+1}\cap V$, integration-by-parts formula~\eqref{eq:DSmoothFunctions} with $\phi = \psi = \xi$ yields
\begin{equation}
\left\langle \mathfrak{D}\xi,\xi \right\rangle_W = (A\xi,\xi) + (\xi,A\xi) =\int\limits_{\partial\Omega}\xi^T A(n)\xi = 2\int\limits_{\partial\Omega}\rho_0c_0\,n\cdot\xi_1\,\xi_2 = 2\int\limits_{\partial\Omega}\rho_0c_0\,Y\,\xi_2^2, 
\label{eq:preV1a}
\end{equation} 
where we in the last step have employed boundary condition~\eqref{eq:IVP_ElEBC}.
Recalling that $Y\geq 0$ on $\partial\Omega$, we conclude by density that
\begin{equation}
\langle \mathfrak{D}\xi,\xi\rangle_W \geq 0~\text{for all } \xi \in V.
\label{eq:VSubCPlus}\tag{V1a}
\end{equation} 
Analogously, for $\tilde{\xi}\in C^1(\bar{\Omega})^{d+1}\cap \tilde{V}$, integration-by-parts formula~\eqref{eq:DSmoothFunctions} with $\phi = \psi = \tilde{\xi}$ yields
\begin{equation}
\langle \mathfrak{D}\tilde{\xi},\tilde{\xi} \rangle_W = (A\tilde{\xi},\tilde{\xi}) + (\tilde{\xi},A\tilde{\xi}) =\int\limits_{\partial\Omega}\tilde{\xi}^T A(n)\tilde{\xi} = 2\int\limits_{\partial\Omega}\rho_0c_0\,n\cdot\tilde{\xi}_1\,\tilde{\xi}_2 = -2\int\limits_{\partial\Omega}\rho_0c_0\,Y\,\tilde{\xi}_2^2, 
\label{eq:preV1b}
\end{equation} 
where we in the last step have employed the adjoint boundary condition~\eqref{eq:adjointBC}.
Recalling that $Y\geq 0$ on $\partial\Omega$, we conclude by density that
\begin{equation}
\langle \mathfrak{D}\tilde{\xi},\tilde{\xi}\rangle_W \leq 0~\text{for all } \xi \in \tilde{V}.
\label{eq:TildeVSubCMinus}\tag{V1b}
\end{equation} 
Conditions~\eqref{eq:VSubCPlus} and~\eqref{eq:TildeVSubCMinus} together form condition~\eqref{eq:V1} in Section~\ref{sec:preliminaries}. 

The following theorem establishes that the spaces $V$ and $\tilde{V}$ are ``orthogonal'' with respect to boundary operator~\eqref{eq:DlE}. 
\begin{thm}
\label{th:CorToRa85th4}
For $\xi\in V$ and $\tilde{\xi}\in\tilde{V}$, $\langle \mathfrak{D}\xi,\tilde{\xi} \rangle_W = \langle \mathfrak{D}\tilde{\xi},{\xi} \rangle_W = 0$.
\end{thm}
\begin{proof}
By Theorem~\ref{th:Ra85th4} and its symmetric analogue, it is sufficient to establish the claim for $\xi\in C^1(\bar{\Omega})^{d+1}\cap V$ and $\tilde{\xi}\in C^1(\bar{\Omega}^{d+1})\cap \tilde{V}$. 
Therefore, for $\xi\in C^1(\bar{\Omega})^{d+1}\cap V$ and $\tilde{\xi}\in C^1(\bar{\Omega}^{d+1})\cap \tilde{V}$, integration-by-parts formula~\eqref{eq:DSmoothFunctions} with $\phi = \xi$ and $\psi = \tilde{\xi}$ yields  
\begin{equation}
\langle \mathfrak{D}{\xi},\tilde{\xi} \rangle_W = (A{\xi},\tilde{\xi}) + ({\xi},A\tilde{\xi}) =\int\limits_{\partial\Omega}{\xi}^T A(n)\tilde{\xi} = 0,
\label{eq:Duv0_1}
\end{equation}
where we in the last step employed Lemma~\ref{lem:basicAnOrth}.   
\end{proof}

Rauch demonstrates that if $\xi\in W$ satisfies 
\begin{equation}
(A\xi,\tilde{\xi}) + (\xi,A\tilde{\xi}) = 0~\text{for all $\xi\in Lip(\bar{\Omega})^{d+1}\cap \tilde{V}$,}
\label{eq:RaPrp3}
\end{equation}
then $\xi\in V$~\cite{Ra85}[Proposition 3]. 
Since $Lip(\bar{\Omega})^{d+1}\cap \tilde{V}\subset \tilde{V}$ an immediate consequence is the following property.
\begin{thm}
\label{th:CorToRa85Prp3}
Assume that $\xi\in W$.
If $\langle \mathfrak{D}\tilde{\xi},{\xi}\rangle_W = 0$ for all $\tilde{\xi}\in\tilde{V}$, then $\xi\in V$. 
\end{thm}

Recall from Section~\ref{sec:preliminaries} the definition of annihilator $\mathfrak{D}(\tilde{V})^0 = \{\xi\in W\mid \langle \mathfrak{D}\tilde{\xi},\xi\rangle_W=0\,\forall\tilde{\xi}\in \tilde{V}\}$. 
Thus, Theorem~\ref{th:CorToRa85th4} implies that $V\subset \mathfrak{D}(\tilde{V})^{0}$, while Theorem~\ref{th:CorToRa85Prp3} implies that $\mathfrak{D}(\tilde{V})^{0}\subset V$, that is,
\begin{equation}
V = \mathfrak{D}(\tilde{V})^0.
\label{eq:V21}\tag{V2a}
\end{equation}
Similarly, $\mathfrak{D}(V)^0 = \{\tilde{\xi}\in  W\mid \langle \mathfrak{D}\xi,\tilde\xi\rangle_W=0\,\forall\xi\in V\}$.
Exploiting, once more, the symmetry in the relations of Lemma~\ref{lem:basicAnOrth}, we conclude by Theorem~\ref{th:CorToRa85th4} and the analogue of Theorem~\ref{th:CorToRa85Prp3} that
\begin{equation}
\tilde{V} = \mathfrak{D}(V)^0.
\label{eq:V22}\tag{V2b}
\end{equation} 
Conditions~\eqref{eq:V21} and~\eqref{eq:V22} combine to  condition~\eqref{eq:V2} in Section~\ref{sec:preliminaries}.
We have thus demonstrated that conditions~\eqref{eq:T1},~\eqref{eq:T2},~\eqref{eq:V1}, and~\eqref{eq:V2} are satisfied.
Thus, Theorem~\ref{th:gen} yields the following well-posedness result. 
\begin{thm}
\label{thm:linEulWP}
If $f \in L^1((0,\tau);L)$ and $\xi_\I\in L$, initial--boundary value problem~\eqref{eq:IVP_ElE} is well-posed in the sense of Theorem~\ref{th:gen}. 
\end{thm}  

\subsection{The Lagrangian displacement is well-defined}
Here, we make the same assumptions on the domain $\Omega$ and on the background flow as in the previous section; in particular, we restrict the attention to the case $n\cdot u_0 = 0$ on $\partial\Omega$.
Employing the abstract framework for Friedrichs' systems, briefly recalled in Section~\ref{sec:preliminaries}, we demonstrate that if we supply an initial condition, the Lagrangian displacement is unambiguously defined by equation~\eqref{eq:dEuRelw}; that is, the Lagrangian displacement is defined as the solution to the initial value problem
\begin{subequations}
	\begin{align}
		(\partial_t + \Lie{u_0})w &= \pert{}u~&\text{in $\Omega$ for all $t\in(0,\tau)$,}\label{eq:IVP_wEQ}\tag{(v.s.~\eqref{eq:dEuRelw})~\ref{eq:IVP_w}a}\\
		w &= w_\I~&\text{ in $\Omega$ at $t = 0,$}\label{eq:IVP_wIC}\tag{\ref{eq:IVP_w}b}	
	\end{align}
	\label{eq:IVP_w}%
\end{subequations}
where $0<\tau<\infty$.

Analogously to the previous section, let $L = L_{\rho_0}(\Omega)^d$ and $\mathcal{D}= C_0^{\infty}(\Omega)^d$.
In equation~(\ref{eq:IVP_w}a) we find the operator $T = u_0\cdot\nabla - \nabla u_0$ and note its formal adjoint in $L$, $\tilde{T} = -u_0\cdot\nabla - (\nabla u_0)^T$.
Then conditions~\eqref{eq:T1} and~\eqref{eq:T2} in Section~\ref{sec:preliminaries} are satisfied. 
In this case the graph space is $W = \{\xi\in L\mid (u_0\cdot\nabla)\xi\in L\}$.
Since $n \cdot u_0 = 0$ on $\partial\Omega$, no boundary condition is needed and we define $V = \tilde{V} = W$.
Using that $C^1(\overline{\Omega})^d$ is dense in $W$~\cite{Ra85}[Proposition 1], we obtain that the boundary operator $\mathfrak{D}\equiv 0$. 
Thus conditions~\eqref{eq:V1} and~\eqref{eq:V2} in Section~\ref{sec:preliminaries} are satisfied. 
From Theorem~\ref{thm:linEulWP}, we obtain that $f:=\pert{}u \in C([0,\tau];L)\subset L^1((0,\tau);L)$.
Therefore, Theorem~\ref{th:gen} yields the following well-posedness result regarding initial value problem~\eqref{eq:IVP_w}. 
\begin{thm}
\label{thm:lagrangianDispWellDef}
Assume that $n\cdot u_0 = 0$ on the boundary. For any initial datum $w_\I\in L$, the Lagrangian displacement, defined as the solution to initial value problem~\eqref{eq:IVP_w}, is well-defined in the sense of Theorem~\ref{th:gen}.
\end{thm}
Recall from Section~\ref{sec:noResAss} that, since the background flow is assumed to be everywhere tangential to the boundary, we may always choose the initial datum $w_\I$ in inital condition~\eqref{eq:IVP_wIC} so that the ``no resonance'' assumption is satisfied.
         
\begin{rmrk}
Note that for background flows that cross the domain boundary, we need to supply \emph{both} an initial condition and a boundary condition on $\Gamma_-$ in order for the Lagrangian displacement to be well-defined by equation~\eqref{eq:dEuRelw}.
Moreover, this general case can be analyzed analogously as in the proof of Theorem~\ref{th:noResWellPo}.
\qed
\end{rmrk}  
  
\section{An energy estimate for Galbrun's equation}
In this section, we derive an a priori energy estimate for Galbrun's equation~\eqref{eq:G}. 
In contrast to Section~\ref{sec:existence} that mostly considers \emph{steady} background flows, we consider \emph{unsteady} background flows in this section.
As will be seen, the obtained energy estimate for Galbrun's equation has the same form as the one for the first order system 
\begingroup
\renewcommand*{\arraystretch}{1.8}
\begin{equation}
\!\!\left(\!\!\rho_0\partial_t \!+\!
\begin{pmatrix}
	\rho_0u_0\!\cdot\!\nabla & \nabla(\rho_0c_0 \cdot) & 0\\
	\rho_0c_0\nabla\cdot & \rho_0u_0\!\cdot\!\nabla & 0\\
	0&0& \rho_0u_0\!\cdot\!\nabla
\end{pmatrix}
\!\!+\!\rho_0c_0\!\!
\begin{pmatrix}
	\dfrac{\nabla u_0}{c_0} & -\dfrac{\nabla\rho_0}{\rho_0} & 0\\
	\dfrac{\nabla\rho_0}{\rho_0}\cdot & -\dfrac{D_0c_0}{c_0^2} & 0\\
	-\dfrac{1}{c_0\tau_0} & 0 & -\dfrac{\nabla u_0}{c_0}
\end{pmatrix}
\!\!\right)
\!\!
\begin{pmatrix}
	\pert{}u\\
	\pert{}\hat{\rho}\\
	\dfrac{w}{\tau_0}
\end{pmatrix}
\!\!=\!\rho_0\!\!
\begin{pmatrix}
	\pert{}\varphi\\
	0\\
	0
\end{pmatrix}\!\!,\!\!
\label{eq:scaledeElE}
\end{equation}
\endgroup
which is formed by appending equation~\eqref{eq:dEuRelw} to the linearized Euler's equations~\eqref{eq:scaledElE}.
In equation~\eqref{eq:scaledeElE} we have introduced the scaled quantity $\tau_0^{-1}w$ (in units of velocity), where, as before, $\tau_0>0$ is an arbitrary time scale. 
To derive the energy estimates, we assume sufficient regularity of the solution and the background flow.
By applying $\int_\Omega(\pert{}u^T, \pert{}\hat{\rho},\tau_0^{-1}w^T)$ to equation~\eqref{eq:scaledeElE} from the left and integrating by parts all terms containing first order derivatives (taking advantage of integration-by-parts formula~\eqref{eq:ip3}), we obtain
\begin{equation}
\begin{aligned}
&\frac{1}{2}\frac{d}{dt}\left(\tau_0^{-2}(\rho_0w,w) + (\rho_0\,\pert{}\hat{\rho},\pert{}\hat{\rho})+(\rho_0\,\pert{}u,\pert{}u)\right) = \\
&=\tau_0^{-2}(\rho_0w,(\nabla u_0)w) + \left(\rho_0\pert{}\hat{\rho},\frac{D_0 c_0}{c_0}\pert{}\hat{\rho}\right) - (\rho_0\pert{}u,(\nabla u_0)\pert{}u) + \tau_0^{-2}(\rho_0\pert{}u,w)\\
&- \frac{1}{2}\int\limits_{\partial\Omega}{\rho_0\xi^T\! K(n)\xi } + (\rho_0\pert{}u,\pert{}\varphi),
\end{aligned}
\label{eq:preEstimateEEle}
\end{equation}
where 
\begin{equation}
\xi^T\! K(n)\xi = 
\begin{pmatrix}
	\pert{}u\\
	\pert{}\hat{\rho}\\
	w/\tau_0
\end{pmatrix}^T\!\!\!
\begin{pmatrix}
n\cdot u_0 & c_0 n & 0\\
c_0 n\cdot & n\cdot u_0 & 0\\
0 & 0 & n\cdot u_0	
\end{pmatrix}
\begin{pmatrix}
	\pert{}u\\
	\pert{}\hat{\rho}\\
	w/\tau_0
\end{pmatrix}.
\label{eq:enEstGlEleBndTerm}
\end{equation}

Our first result towards an energy estimate for Galbrun's equation is to show that relation~\eqref{eq:preEstimateEEle} also holds for Galbrun's equation. 
However, in that case the ``no resonance" assumption holds and $\pert{}\hat{\rho} = c_0\rho_0^{-1}\pert{}\rho = -c_0\rho_0^{-1}\nabla\cdot(\rho_0w)$ (recall relation~\eqref{eq:noResIndErho} and definition~\eqref{eq:dHatRho}). 
 
\begin{lem}
\label{lem:enEst}
If $w$ is sufficiently regular and satisfies Galbrun's equation~\eqref{eq:G} for a sufficiently regular background flow, then relation~\eqref{eq:preEstimateEEle} holds with $\pert{}u = (\partial_t+\Lie{u_0})w$ and $\pert{}\hat{\rho} = -c_0\rho_0^{-1}\nabla\cdot(\rho_0 w)$. 
\end{lem} 
\begin{proof}
We start by writing Galbrun's equation~\eqref{eq:G} as the first order system
\begin{subequations}
\begin{align}
\rho_0{D}_0\,\pert{}u + \nabla(\rho_0c_0\,\pert{}\hat{\rho}) + \rho_0(\pert{}u\cdot\nabla){u_0}-\rho_0\frac{\nabla p_0}{\rho_0c_0}\pert{}\hat{\rho} &= \rho_0\,\pert{}{\varphi},\label{eq:G1stOrder1}\\
\pert{}u &= (\partial_t+\Lie{u_0})w =  D_0 w - (w\cdot\nabla)u_0,\label{eq:G1stOrder2}\\
\pert{}\hat{\rho} &= -c_0\rho_0^{-1}\nabla\cdot(\rho_0 w).\label{eq:G1stOrder3}
\end{align}
\label{eq:G1stOrder}%
\end{subequations}
Multiplying equation~\eqref{eq:G1stOrder1} with $\pert{}u^T$ from the left and integrating over $\Omega$, we obtain 
\begin{equation}
(\pert{}u,\rho_0\,\pert{}\varphi) = (\pert{}u,\rho_0D_0\,\pert{}u) + (\pert{}u, \nabla(\rho_0c_0\,\pert{}\hat{\rho})) + (\pert{}u,\rho_0(\pert{}u\cdot\nabla)u_0)  - (\pert{}u,c_0\nabla\rho_0\,\pert{}\hat{\rho}),
\label{eq:prp1_1}
\end{equation}
where we have rewritten the last term using that $\nabla p_0 = c_0^2\nabla \rho_0$, which follows from the equation of state~\eqref{eq:Eu03}.
Applying integration-by-parts formula~\eqref{eq:ip3} to the first term in equation~\eqref{eq:prp1_1}, we find that
\begin{equation}
(\pert{}u,\rho_0D_0\,\pert{}u) = \frac{1}{2}\frac{d}{dt}(\rho_0\,\pert{}u,\pert{}u) + \frac{1}{2}(\rho_0(n\cdot u_0)\,\pert{}u,\pert{}u)_{\partial\Omega}.
\label{eq:prp1_2}
\end{equation}   
Integration by parts of the second term of equation~\eqref{eq:prp1_1} yields
\begin{equation}
(\pert{}u,\nabla(\rho_0c_0\,\pert{}\hat{\rho})) = -(\nabla\cdot\pert{}u,\rho_0c_0\,\pert{}\hat{\rho}) + (n\cdot\pert{}u,\rho_0c_0\,\pert{}\hat{\rho})_{\partial\Omega}.
\label{eq:prp1_3}
\end{equation}   
From identity~\eqref{eq:divdEu} and definition~\eqref{eq:G1stOrder3}, we deduce that
\begin{align}
-\nabla\cdot\pert{}u &= -\nabla\cdot\left(\rho_0^{-1}\rho_0\,\pert{}u\right) = -\rho_0\,\pert{}u\cdot\nabla\rho_0^{-1} - \rho_0^{-1}\nabla\cdot(\rho_0\,\pert{}u)\nonumber\\ 
&=  \rho_0^{-1}\pert{}u\cdot\nabla\rho_0 + D_0(c_0^{-1}\pert{}\hat{\rho}) = \rho_0^{-1}\pert{}u\cdot\nabla\rho_0 + (D_0c_0^{-1})\pert{}\hat{\rho} + c_0^{-1}D_0\,\pert{}\hat{\rho}, 
\label{eq:prp1_4}
\end{align}
which substituted into expression~\eqref{eq:prp1_3} gives
\begin{align}
(\pert{}u,\nabla(\rho_0c_0\,\pert{}\hat{\rho})) &= (\rho_0^{-1}\pert{}u\cdot\nabla\rho_0 + (D_0c_0^{-1})\pert{}\hat{\rho} + c_0^{-1}D_0\,\pert{}\hat{\rho},\rho_0c_0\,\pert{}\hat{\rho}) + (n\cdot\pert{}u,\rho_0c_0\,\pert{}\hat{\rho})_{\partial\Omega}\nonumber\\
&= (\rho_0^{-1}\pert{}u\cdot\nabla\rho_0,\rho_0c_0\,\pert{}\hat{\rho}) + ((D_0c_0^{-1})\pert{}\hat{\rho},\rho_0c_0\,\pert{}\hat{\rho}) + (\rho_0D_0\,\pert{}\hat{\rho},\pert{}\hat{\rho}) + (n\cdot\pert{}u,\rho_0c_0\,\pert{}\hat{\rho})_{\partial\Omega}\nonumber\\
&= (\rho_0^{-1}\pert{}u\cdot\nabla\rho_0,\rho_0c_0\,\pert{}\hat{\rho}) + ((D_0c_0^{-1})\pert{}\hat{\rho},\rho_0c_0\,\pert{}\hat{\rho}) + \frac{1}{2}\frac{d}{dt}(\rho_0\,\pert{}\hat{\rho},\pert{}\hat{\rho})\nonumber\\ 
&\phantom{=}+ \frac{1}{2}(\rho_0(n\cdot u_0)\pert{}\hat{\rho},\pert{}\hat{\rho})_{\partial\Omega} + (n\cdot\pert{}u,\rho_0c_0\,\pert{}\hat{\rho})_{\partial\Omega},
\label{eq:prp1_5}
\end{align} 
where we have employed integration-by-parts formula~\eqref{eq:ip3} to arrive at the final expression.
Substituting expressions~\eqref{eq:prp1_2} and~\eqref{eq:prp1_5} into expression~\eqref{eq:prp1_1}, we find that
\begin{align}
(\pert{}u,\rho_0\,\pert{}\varphi) &=\frac{1}{2}\frac{d}{dt}\Big((\rho_0\,\pert{}u,\pert{}u) + (\rho_0\,\pert{}\hat{\rho},\pert{}\hat{\rho})\Big) +  (\rho_0^{-1}\pert{}u\cdot\nabla\rho_0,\rho_0c_0\,\pert{}\hat{\rho}) + ((D_0c_0^{-1})\pert{}\hat{\rho},\rho_0c_0\,\pert{}\hat{\rho})\nonumber\\ &\phantom{=}+ (\pert{}u,\rho_0(\pert{}u\cdot\nabla)u_0)  - (\pert{}u,c_0\nabla\rho_0\,\pert{}\hat{\rho})\nonumber\\ 
&\phantom{=}+ \frac{1}{2}(\rho_0(n\cdot u_0)\pert{}u,\pert{}u)_{\partial\Omega} + \frac{1}{2}(\rho_0(n\cdot u_0)\pert{}\hat{\rho},\pert{}\hat{\rho})_{\partial\Omega} + (n\cdot\pert{}u,\rho_0c_0\,\pert{}\hat{\rho})_{\partial\Omega}.
\label{eq:prp1_6}
\end{align}

Multiplying equation~\eqref{eq:G1stOrder2} by $\rho_0 w^T$ from the left and integrating over $\Omega$, we obtain that 
\begin{equation}
(\rho_0w,\pert{}u) = (\rho_0w,D_0w) - (\rho_0w,(w\cdot\nabla)u_0).
\label{eq:prp1_7}
\end{equation}
Integration-by-parts formula~\eqref{eq:ip3} applied to the first term in equation~\eqref{eq:prp1_7} gives
\begin{equation}
\tau_0^{-2}(\rho_0w,\pert{}u) = \tau_0^{-2}\frac{1}{2}\frac{d}{dt}(\rho_0w,w) + \frac{1}{2}\tau_0^{-2}(\rho_0(n\cdot u_0)w,w)_{\partial\Omega} - \tau_0^{-2}(\rho_0w,(w\cdot\nabla)u_0),
\label{eq:prp1_8}
\end{equation}  
where the factor $\tau_0^{-2}$ has been introduced to match the dimensions of expressions~\eqref{eq:prp1_6} and~\eqref{eq:prp1_8}.
By adding equations~\eqref{eq:prp1_6} and~\eqref{eq:prp1_8}, we obtain relation~\eqref{eq:preEstimateEEle}.
\end{proof}

We will once more restrict our attention to the case where the background flow is everywhere tangential to $\partial\Omega$, that is, the case when the ``no resonance'' assumption can be enforced by adjusting the initial datum of the Lagrangian displacement.
To derive an energy estimate for Galbrun's equation~\eqref{eq:G} in this case, we first completely specify the initial--boundary value problem
\begin{subequations}
	\begin{align}
		\rho_0D_0(D_0w - (w\cdot\nabla)u_0) -\nabla(c_0^2\nabla\cdot(\rho_0w)) \,+\nonumber\\
		\rho_0\left((D_0w - (w\cdot\nabla)u_0)\cdot\nabla\right)u_0+\frac{\nabla p_0}{\rho_0}\nabla\cdot(\rho_0w) &= \rho_0\pert{}\varphi~&\text{in $\Omega$ for all $t\in(0,\tau),$}\label{eq:IVP_GEq}\\
		w = w_\I~\text{and } \partial_t w &= z_\I~&\text{ in $\Omega$ at $t = 0,$}\label{eq:IVP_GIC}\\
		-Yc_0\rho_0^{-1}\nabla\cdot(\rho_0 w)-n\cdot((\partial_t+\Lie{u_0})w) &= g~&\text{on $\partial\Omega$ for all $t\in(0,\tau)$,}\label{eq:IVP_GBC}
	\end{align}
	\label{eq:IVP_G}%
\end{subequations}
where, as in Section~\ref{sec:ElE_WP}, $Y:\partial\Omega\to [0,\infty)$ is a Lipschitz continuous (dimensionless) admittance function with the additional requirement that $Y\geq a>0$, and $0<\tau<\infty$. 
We assume that 
\begin{align}
&0 < \underline{\rho_0} := \inf_{\bar{\Omega}\times [0,\tau]}\rho_0 \leq \sup_{\bar{\Omega}\times [0,\tau]}\rho_0 =: \overline{\rho_0} <\infty,\label{eq:b1}\\
&0<\underline{c_0} := \inf_{\bar{\Omega}\times [0,\tau]}c_0 \leq \sup_{\bar{\Omega}\times [0,\tau]}c_0 =: \overline{c_0} < \infty,\label{eq:b2}\\
&\sup_{t\in [0,\tau]}\|\nabla u_0(t)\|_{L^{\infty}(\Omega)^{d\times d}} =: \overline{|\nabla u_0|}<\infty,\label{eq:b3}\\
&\sup_{t\in[ 0,\tau]}\|(c_0^{-1}D_0c_0)(t)\|_{L^{\infty}(\Omega)} =: \overline{|c_0^{-1}D_0c_0|}<\infty. \label{eq:b4}
\end{align}

\begin{thm}
\label{prp:enEstG} 
Assume that $w$ is a sufficiently regular solution to initial boundary value problem~\eqref{eq:IVP_G} for sufficiently regular data and a sufficiently regular background flow that is everywhere tangential to the boundary.
For any finite time $\tau>0$, there exists $C,\nu > 0$ that are independent of $w$ such that for any $0<t<\tau$
\begin{equation}
\begin{aligned}
&\tau_0^{-2}\|{w}(t)\|^2 + \|{c_0\rho_0^{-1}\nabla\cdot(\rho_0w)}(t)\|^2 + \|(\partial_t+\mathcal{L}_{u_0})w(t)\|^2 \leq \\
&\quad Ce^{\nu t} \Bigg(\tau_0^{-2}\|{w}_\I\|^2 + \|c_0(0)\rho_0(0)^{-1}\nabla\cdot(\rho_0(0)w_\I)\|^2 + \|z_\I+\mathcal{L}_{u_0(0)}w_\I\|^2 +\\ 
&\quad\quad\int\limits_{0}^{t}{\!\tau_0\|\pert{}\varphi(s)\|^2 + a^{-1}\overline{c_0}\|g(s)\|_{\partial\Omega}^2\,\mathrm{d}s}\Bigg).
\end{aligned}
\label{eq:enEstG}
\end{equation}
\end{thm}
\begin{proof}
As in the proof of Lemma~\ref{lem:enEst}, we introduce $\pert{}u$ and $\pert{}\hat{\rho}$ by expressions~\eqref{eq:G1stOrder2} and~\eqref{eq:G1stOrder3}. 
Moreover, we introduce the corresponding expressions at time $t = 0$
\begin{equation}
\begin{aligned}
\pert{}u_\I = z_\I+\Lie{u_0(0)}w_\I,\\
\pert{}\hat{\rho}_\I = -c_0(0)\rho_0(0)^{-1}\nabla\cdot(\rho_0(0) w_\I).
\end{aligned}
\label{eq:du0_drho0}
\end{equation}
Then estimate~\eqref{eq:enEstG} takes the form 
\begin{equation}
\begin{aligned}
&\tau_0^{-2}\|{w}(t)\|^2+\|\pert{}\hat{\rho}(t)\|^2+\|\pert{}u(t)\|^2 \leq \\
&\quad\quad Ce^{\nu t}\left(\tau_0^{-2}\|{w}_\I\|^2 + \|\pert{}\hat{\rho}_\I\|^2 + \|\pert{}u_\I\|^2 + \int\limits_{0}^{t}{\tau_0\|\pert{}\varphi(s)\|^2 + \overline{c_0}\|g(s)\|_{\partial\Omega}^2\,\mathrm{d}s}\right).
\label{eq:enEstG1}
\end{aligned}
\end{equation}
For convenience and consistent with definition~\eqref{eq:enEstGlEleBndTerm}, we introduce $\xi = (\pert{}u,\pert{}\hat{\rho},\tau_0^{-1}w)$ and $\xi_\I = (\pert{}u_\I,\pert{}\hat{\rho}_\I,\tau_0^{-1}w_\I)$. 
By the bounds~\eqref{eq:b3} and~\eqref{eq:b4}, we obtain from relation~\eqref{eq:preEstimateEEle} that
\begin{equation}
\frac{1}{2}\frac{d}{dt}\|\xi\|_{\rho_0}^2 
\leq \underbrace{\tau_0^{-1}\max\left\{1+\tau_0\overline{|\nabla u_0|},\tau_0\overline{|c_0^{-1}D_0c_0|}\right\}}_{=: \nu/2}\|\xi\|_{\rho_0}^2 + \frac{\tau_0}{2}\|\pert{}\varphi\|_{\rho_0}^2-\frac{1}{2}\int\limits_{\partial\Omega}\rho_0 \xi^T\! K(n)\xi,
\label{eq:preEst1}
\end{equation}
where $\|\cdot\|_{\rho_0}^2 = (\rho_0\cdot,\cdot)$. 
We now derive an estimate for the boundary term in expression~\eqref{eq:preEst1}. 
Observe that by expressions~\eqref{eq:G1stOrder2} and~\eqref{eq:G1stOrder3}, boundary condition~\eqref{eq:IVP_GBC} can be reexpressed in $\pert{}u$ and $\pert{}\hat{\rho}$,
\begin{equation}
-n\cdot\pert{}u + Y\pert{}\hat{\rho} = g~\text{at $\partial\Omega$ for all $t\in (0,\tau)$}.
\label{eq:GBCdu_drho}
\end{equation}
Recalling that $n\cdot u_0 = 0$ on $\partial\Omega$, we obtain by boundary condition~\eqref{eq:GBCdu_drho} and definition~\eqref{eq:enEstGlEleBndTerm} that
\begin{align}
-\frac{1}{2}\rho_0\xi^T\! K(n)\xi &= -\rho_0c_0(n\cdot\pert{}u)\pert{}\hat{\rho} = -\rho_0c_0Y\,\pert{}\hat{\rho}^2 + \rho_0c_0g\,\pert{}\hat{\rho}\nonumber\\ 
&\leq \frac{1}{2}\rho_0c_0a^{-1}g^2 + \rho_0c_0\left(\frac{a}{2}-Y\right)\pert{}\hat{\rho}^2 \leq a^{-1}\frac{1}{2}\rho_0c_0 g^2,
\label{eq:bndEst}
\end{align}
where we in the last step have used that $Y\geq a > 0$ on $\partial\Omega$.
Combining estimates~\eqref{eq:preEst1}, and~\eqref{eq:bndEst}, we obtain
\begin{align}
\frac{d}{dt}\|\xi\|_{\rho_0}^2 &\leq \nu\|\xi\|_{\rho_0}^2 + \tau_0\|\pert{}\varphi\|_{\rho_0}^2\!\!+a^{-1}\int\limits_{\partial\Omega}{\rho_0c_0 g^2}\nonumber\\
&\leq \nu \|\xi\|_{\rho_0}^2 + \tau_0\overline{\rho_0}\|\pert{}\varphi\|^2\!\!+a^{-1}\overline{\rho_0}\,\overline{c_0}\|g\|_{\partial\Omega}^2,
\label{eq:preEst2}
\end{align}
where we in the last step have used the bounds~\eqref{eq:b1} and~\eqref{eq:b2}.
Multiplying estimate~\eqref{eq:preEst2} by the integrating factor $\exp(-\nu t)$, integrating over the time interval $(0,t)\subset [0,\tau]$, and using that $\xi(0) = \xi_\I$, we obtain
\begin{equation}
\|\xi(t)\|_{\rho_0}^2 \leq e^{\nu t}\|\xi_\I\|_{\rho_0(0)}^2 + \overline{\rho_0}\int\limits_{0}^{t}{e^{\nu_\tau(t-s)}\left(\tau_0\|\pert{}\varphi\|^2\!\!+a^{-1}\overline{c_0}\|g\|_{\partial\Omega}^2\right)\mathrm{d}s}.
\label{preEst3} 
\end{equation}
Note that the bounds~\eqref{eq:b1} imply that $\underline{\rho_0}\|\cdot\|\leq \|\cdot\|_{\rho_0(t)}\leq \overline{\rho_0}\|\cdot\|$ for all $t\in[0,\tau]$. 
Thus, using $\exp(\nu(t-s))\leq \exp(\nu t)$, we end up with estimate~\eqref{eq:enEstG} with $C = \overline{\rho_0}/\underline{\rho_0}$.  
\end{proof} 

\section{Discussion}
The abstract definition~\eqref{eq:dEuRelw} of the Lagrangian displacement calls for adequate initial and boundary data. 
In the special case when the background flow is everywhere tangential to the boundary, no boundary datum is needed, and the initial datum can be determined so that the ``no resonance'' assumption is satisfied.
However, when the background flow passes through the boundary of the domain, we do not know whether a boundary datum that guarantees fulfillment of the ``no resonance'' assumption can be provided for the Lagrangian displacement.

We note that fulfillment of the ``no resonance'' assumption in any case requires that the initial data of $\pert{}\rho$ and $w$ satisfy condition~\eqref{eq:noResHomIC}; that is, the ``no resonance'' assumption imposes a restriction on the initial datum of the Lagrangian displacement. 
However, at least when the flow is everywhere tangential to the boundary, the ``no resonance'' assumption imposes no restriction on $\pert{}u$ and $\pert{}\rho$, that is, it does not restrict the physical behavior of the system. 

The role of the ``no resonance'' assumption has a striking resemblance in electrodynamics. 
In vacuum, Maxwell's equations are given by
\begin{subequations}
\begin{align}
\partial_t B + \nabla\times E &= 0,\label{eq:Me1}\\
\mu_0\epsilon_0\partial_t E - \nabla\times B &= -\mu_0 J,\label{eq:Me2}\\
\nabla\cdot B &= 0,\label{eq:Me3}\\
\epsilon_0 \nabla\cdot E &= \chi,\label{eq:Me4}
\end{align} 
\label{eq:Me}%
\end{subequations}
where $E$ and $B$ denotes the electric and magnetic fields, $J$ and $\chi$ the current and charge densities, $\mu_0$ the vacuum permeability, and $\epsilon_0$ the vacuum permittivity~\cite{Mo03}[Chapter 1.2].
It turns out that the divergence conditions~\eqref{eq:Me3} and~\eqref{eq:Me4} can be regarded as consequences of the principle of conservation of electric charge and equations~\eqref{eq:Me1} and~\eqref{eq:Me2}~\cite{Mo03}[Chapter 1.2].
Indeed, applying the divergence to equations~\eqref{eq:Me1} and~\eqref{eq:Me2}, we obtain that
\begin{subequations}
\begin{align}
\partial_t \nabla \cdot B &= 0,\label{eq:divMe1}\\
\mu_0\epsilon_0\partial_t \nabla\cdot E = -\mu_0 \nabla\cdot J = \mu_0\partial_t \chi \implies \partial_t(\epsilon_0\nabla\cdot E - \chi) &= 0,\label{eq:divMe2}
\end{align} 
\label{eq:divMe1And2}%
\end{subequations}
where we in expression~\eqref{eq:divMe2} have used that $\partial_t\chi + \nabla\cdot J = 0$, which expresses conservation of electric charge.
Thus, if the initial data for $E$ and $B$ satisfy divergence conditions~\eqref{eq:Me3} and~\eqref{eq:Me4}, then divergence conditions~\eqref{eq:Me3} and~\eqref{eq:Me4} hold for all subsequent times.
Analogously for Galbrun's equation, relation~\eqref{eq:LlE2Euler} was derived by applying the divergence to definition~\eqref{eq:dEuRelw} and invoking equation~\eqref{eq:ElE2}, which here plays the same role as conservation of charge. 
If the background flow is everywhere tangential to the boundary of the domain and the initial datum for $w$ satisfies divergence condition~\eqref{eq:noResIndErho}, then divergence condition~\eqref{eq:noResIndErho} holds for all subsequent times.
In principle, if the background flow passes through the boundary and $w$ satisfies divergence condition~\eqref{eq:noResIndErho} in $\Omega$ at $t = 0$ \emph{and} on $\Gamma_-\subset\partial\Omega$ for all $t>0$, then divergence condition~\eqref{eq:noResIndErho} holds for all subsequent times.
As already pointed out above, however, it is not clear how to handle the extra condition on the boundary part $\Gamma_-$.     

It would be tempting to define the Lagrangian displacement by both relations~\eqref{eq:dEuRelw} and~\eqref{eq:noResIndErho}, since then the ``no resonance'' assumption would be automatically satisfied.
However, such a definition does not fit the Friedrichs' framework employed here, and at present we do not know how it should be handled.    
Friedman and Schutz~\cite{FrSc78a} have made some investigations into the matter and remark that if $w$ satisfies both relations~\eqref{eq:dEuRelw} and~\eqref{eq:noResIndErho}, then so does $w + \rho_0^{-1}\nabla\times v$ for any vector field $v$ satisfying $\partial_tv + (u_0\cdot\nabla)v + (\nabla u_0)^Tv = 0$.
Thus, even when a solution that satisfies both relations~\eqref{eq:dEuRelw} and~\eqref{eq:noResIndErho} can be found, one may need to deal with a possible non-uniqueness.        
   
We have presented a mildly well-posed initial--boundary value problem for linearized Euler's equations~\eqref{eq:IVP_ElE} in the special case that the background flow is everywhere tangential to the boundary.
Given a solution to that initial--boundary value problem, we may define the Lagrangian displacement by~\eqref{eq:dEuRelw} such that relation~\eqref{eq:noResIndErho} holds in $\Omega$ at $t = 0$.
However, we cannot rigorously conclude that relation~\eqref{eq:noResIndErho} holds for all subsequent times and thereby that our Lagrangian displacement satisfies Galbrun's equation~\eqref{eq:G}.
The issue is that the derivation of relation~\eqref{eq:LlE2Euler} appears to require more regularity of $\pert{}u, \pert{}\rho$ and $w$ than what we obtain, at least without performing further analysis.
Regularity is also an issue for the energy estimate in Theorem~\ref{prp:enEstG} since it hinges on the existence of sufficiently regular solutions to initial--boundary value problem~\eqref{eq:IVP_G}.
We expect that resolving the regularity issue to be challenging and leave this matter open for future investigation.   
 
Although we have not performed any numerical experiments, the ideas behind our analysis could be transformed into a numerical scheme for solving Galbrun's equation. 
The resulting scheme would not be computationally attractive since the linearized Euler's equations need to be solved for the Eulerian perturbations before the Lagrangian displacement can be determined from definition~\eqref{eq:dEuRelw}.
To make things worse, it is typically not the Lagrangian displacement but rather the Eulerian perturbations that are the unknowns of interest.
Nevertheless, it would be interesting to compare such indirect numerical scheme to other more direct approaches for solving Galbrun's equation or regularized formulations of Galbrun's equation.  

\section{Acknowledgements}
The authors thank Professor Rainer Picard for helpful discussions and valuable feedback.

This work is financially supported by the Swedish Research Council (grant numbers  2013-03706 and 2018-03546), and eSSENCE, a strategic collaborative eScience program funded by the Swedish Research Council.

\appendix
\section{The Lagrangian displacement}
\label{sec:LagDisp}
In this section, we present the physical motivation behind definition~\eqref{eq:dEuRelw} of the Lagrangian displacement. 
Consider a fluid permeating all of space, where 
$u(x,t)\in\mathbb{R}^d$ gives the instantaneous fluid velocity at position $x\in\mathbb{R}^d$ and time $t>0$.
The location at time $t>0$, $X(p,t)$, of a massless fluid particle initially located at position $p\in\mathbb{R}^3$ can be found by solving the initial value problem
\begin{equation}
\begin{aligned}
\dot{X}(p,t) &= u(X(p,t),t)\enspace t>0\\
X(p,0) &= p,  
\end{aligned}
\label{eq:X}
\end{equation} 
where the derivative is with respect to time. 
In the background flow, a fluid particle initially located at position $p\in\mathbb{R}^3$ would follow the path $X_0(p,\cdot)$ given by the solution of the initial value problem   
\begin{equation}
\begin{aligned}
\dot{X}_0(p,t) &= u_0(X_0(p,t),t)\enspace t>0\\
X_0(p,0) &= p,  
\end{aligned}
\label{eq:X0}
\end{equation}
which is the analogue of problem~\eqref{eq:X}.
Recall from Section~\ref{sec:derivation} that $u$ and $u_0$ are related through $u = u_0 +\pert{}u$, where, as before, $\pert{}u$ denotes the Eulerian perturbation of the fluid velocity.
For each $p$ and all times $t\geq 0$ we define the Lagrangian displacement as the displacement of individual fluid particles 
\begin{equation}
W(p,t) = X(p,t)-X_0(p,t).
\label{eq:W}
\end{equation}
Differentiating definition~\eqref{eq:W} with respect to time and using expressions~\eqref{eq:X},~\eqref{eq:X0}, and~\eqref{eq:W}, we find that
\begin{equation}
\begin{aligned}
\dot{W}(p,t) &= u(X(p,t),t) - u_0(X_0(p,t),t)\\ 
&= u(X_0(p,t)+W(p,t),t) - u_0(X_0(p,t),t),\\
W(p,0) &= 0.
\end{aligned}
\label{eq:dtW}
\end{equation}
We assume that for each $t>0$ the mapping $X_0(\cdot,t)$ is a diffeomorphism and define the vector field $w:\mathbb{R}^d\times [0,\infty)\to \mathbb{R}^d$ such that
\begin{equation}
W(p,t) = w(X_0(p,t),t). 
\label{eq:w}
\end{equation}
That is, $w$ is the Eulerian description of the Lagrangian displacement $W$. 
We note that $\dot{W}(p,t) = (\partial_t + u_0(x,t)\cdot\nabla)w(x,t)\rvert_{x = X_0(p,t)}\equiv D_0w(x,t)\rvert_{x = X_0(p,t)}$. Hence, we may write expression~\eqref{eq:dtW} as
\begin{equation}
\begin{aligned}
D_0w(x,t)\rvert_{x = X_0(p,t)} &= \left[u(x+w(x,t),t) - u_0(x,t)\right]\rvert_{x = X_0(p,t)},\\
w(x,0) &= 0.
\end{aligned}
\label{eq:dtWEuler}
\end{equation}
Thus, to first order in $w$ and $\pert{}u$
\begin{equation}
D_0w = u + (w\cdot\nabla)u_0-u_0 = \pert{}u + (w\cdot\nabla) u_0,
\label{eq:D0w}
\end{equation} 
or equivalently
\begin{equation}
(\partial_t + \mathcal{L}_{u_0})w = \pert{}u, 
\tag{v.s. \eqref{eq:dEuRelw}}
\end{equation}
where $\mathcal{L}_{u_0}w = (u_0\cdot\nabla)w - (w\cdot\nabla)u_0 = -\mathcal{L}_wu_0$ denotes the Lie derivative of $w$ with respect to $u_0$.

\section{Identities for the Lie derivative}
\label{sec:Lie}
In the following, $u, v$ are vector fields and $p$ is a scalar field.
\begin{align}
\Lie{u}p &:= (u\cdot\nabla)p\label{eq:Lieup}\\ 
\Lie{u}v &:= (u\cdot\nabla)v - (v\cdot\nabla)u = -\Lie{v}u\label{eq:Lieuv}\\ 
\nabla\cdot\Lie{u}v &= \Lie{u}(\nabla\cdot v) - \Lie{v}(\nabla\cdot u) = (u\cdot\nabla)(\nabla\cdot v) - (v\cdot\nabla)(\nabla\cdot u)\label{eq:divLieuv}\\
\Lie{u}pv &= v\Lie{u}p + p\Lie{u}v\label{eq:LieProductScalarAndVector}
\end{align} 

\section{Integration-by-parts formula}
In the following $p$ denotes a scalar field (or a Cartesian component of a vector field).
We note that 
\begin{equation}
\frac{d}{dt}(\rho_0p,p) = (\dot{\rho}_0p,p) + 2(\rho_0\dot{p},p),
\label{eq:ip1}
\end{equation}
and
\begin{equation}
(\rho_0(u_0\cdot\nabla)p,p) = -(\nabla\cdot(\rho_0u_0)p,p) - \underbrace{(p,\rho_0(u_0\cdot\nabla)p)}_{=(\rho_0(u_0\cdot\nabla)p,p)} + (\rho_0(n\cdot u_0)p,p)_{\partial\Omega}.
\label{eq:ip2}
\end{equation}
By combining relations~\eqref{eq:ip1} and~\eqref{eq:ip2} and using mass conservation~\eqref{eq:Eu02}, we obtain
\begin{equation}
(\rho_0D_0p,p) = \frac{1}{2}\frac{d}{dt}(\rho_0p,p) + \frac{1}{2}(\rho_0(n\cdot u_0)p,p)_{\partial\Omega}. 
\label{eq:ip3}
\end{equation}


\begin{thebibliography}{10}

\bibitem{AnBu10}
Nenad Antoni{\'c} and Kre{\v{s}}imir Burazin.
\newblock Intrinsic boundary conditions for {Friedrichs} systems.
\newblock {\em Communications in Partial Differential Equations},
  35(9):1690--1715, 2010.

\bibitem{Be06}
Kamel Berriri.
\newblock {\em Approche analytique et num{\'e}rique pour l'a{\'e}roacoustique
  en r{\'e}gime transitoire par le mod{\`e}le de {Galbrun}}.
\newblock PhD thesis, ENSTA ParisTech, 2006.

\bibitem{BeBoJo06}
Kamel Berriri, A-S Dhia, and Patrick Joly.
\newblock Numerical analysis of time-dependent {Galbrun} equation in an
  infinite duct.
\newblock {\em arXiv preprint math/0603546}, 2006.

\bibitem{BoBeJo06}
Anne~Sophie Bonnet-Bendhia, Kamel Berriri, and Patrick Joly.
\newblock R{\'e}gularisation de l'{\'e}quation de {Galbrun} pour
  l'a{\'e}roacoustique en r{\'e}gime transitoire.
\newblock {\em Revue Africaine de la Recherche en Informatique et
  Math{\'e}matiques Appliqu{\'e}es}, 5:65--79, 2006.

\bibitem{Br11}
J-Ph Brazier.
\newblock Derivation of an exact energy balance for {Galbrun} equation in
  linear acoustics.
\newblock {\em Journal of Sound and Vibration}, 330(12):2848--2868, 2011.

\bibitem{BuEr16}
Kre{\v{s}}imir Burazin and Marko Erceg.
\newblock Non-stationary abstract {Friedrichs} systems.
\newblock {\em Mediterranean journal of mathematics}, 13(6):3777--3796, 2016.

\bibitem{ChDu18}
Juliette Chabassier and Marc Durufl{\'e}.
\newblock Solving time-harmonic {Galbrun's} equation with an arbitrary flow.
  {Application} to helioseismology.
\newblock Technical report, INRIA Bordeaux, 2018.

\bibitem{BoDuLeMe07}
Anne-Sophie Bonnet-Ben Dhia, Eve-Marie Duclairoir, Guillaume Legendre, and
  Jean-Fran{\c{c}}ois Mercier.
\newblock Time-harmonic acoustic propagation in the presence of a shear flow.
\newblock {\em Journal of Computational and Applied Mathematics},
  204(2):428--439, 2007.

\bibitem{BoLeLu01}
Anne-Sophie Bonnet-Ben Dhia, Guillaume Legendre, and {\'E}ric Lun{\'e}ville.
\newblock Analyse math{\'e}matique de l'{\'e}quation de {Galbrun} en
  {\'e}coulement uniforme.
\newblock {\em Comptes Rendus de l'Acad{\'e}mie des Sciences-Series
  IIB-Mechanics}, 329(8):601--606, 2001.

\bibitem{BoLeLu03}
Anne-Sophie Bonnet-Ben Dhia, Guillaume Legendre, and Eric Lun{\'e}ville.
\newblock Regularization of the time-harmonic {Galbrun’s} equations.
\newblock In {\em Mathematical and Numerical Aspects of Wave Propagation WAVES
  2003}, pages 78--83. Springer, 2003.

\bibitem{BoMeMiPe10}
AS~Bonnet-Ben Dhia, Jean-Fran{\c{c}}ois Mercier, Florence Millot, and
  S{\'e}bastien Pernet.
\newblock A low-mach number model for time-harmonic acoustics in arbitrary
  flows.
\newblock {\em Journal of Computational and Applied Mathematics},
  234(6):1868--1875, 2010.

\bibitem{BoMeMiPePe12}
AS~Bonnet-Ben Dhia, Jean-Fran{\c{c}}ois Mercier, Florence Millot, S{\'e}bastien
  Pernet, and Emilie Peynaud.
\newblock Time-harmonic acoustic scattering in a complex flow: a full coupling
  between acoustics and hydrodynamics.
\newblock {\em Communications in Computational Physics}, 11(2):555--572, 2012.

\bibitem{DySc79}
J.~Dyson and B.~F. Schutz.
\newblock Perturbations and stability of rotating stars. {I.} {Completeness} of
  normal modes.
\newblock {\em Proceedings of the Royal Society of London. Series A,
  Mathematical and Physical Sciences}, 368(1734):389--410, 1979.

\bibitem{ErGu06a}
A.~Ern and J.-L. Guermond.
\newblock Discontinuous {Galerkin} methods for {Friedrichs'} systems. {I.}
  {General} theory.
\newblock {\em SIAM Journal on Numerical Analysis}, 44(2):753--778, 2006.

\bibitem{ErGuCa07}
Alexandre Ern, Jean-Luc Guermond, and Gilbert Caplain.
\newblock An intrinsic criterion for the bijectivity of {Hilbert} operators
  related to {Friedrich'} systems.
\newblock {\em Communications in Partial Differential Equations},
  32(2):317--341, 2007.

\bibitem{Fe13}
Xue Feng.
\newblock {\em Mod{\'e}lisation num{\'e}rique par {\'e}l{\'e}ments finis d'un
  probl{\`e}me a{\'e}roacoustique en r{\'e}gime transitoire: application {\`a}
  l'{\'e}quation de {Galbrun}}.
\newblock PhD thesis, Universit{\'e} de Technologie de Compi{\`e}gne, 2013.

\bibitem{FeBeBa16}
Xue Feng, Mabrouk Ben~Tahar, and Ryan Baccouche.
\newblock The aero-acoustic {Galbrun} equation in the time domain with
  perfectly matched layer boundary conditions.
\newblock {\em The Journal of the Acoustical Society of America},
  139(1):320--331, 2016.

\bibitem{FrSc78b}
John~L Friedman and Bernard~F Schutz.
\newblock Secular instability of rotating {Newtonian} stars.
\newblock {\em Astrophysical Journal}, 222:281--296, 1978.

\bibitem{FrRo60}
E.~Frieman and M.~Rotenberg.
\newblock On hydromagnetic stability of stationary equilibria.
\newblock {\em Rev. Mod. Phys.}, 32:898--902, Oct 1960.

\bibitem{Ga03}
Gw{\'e}na{\"e}l Gabard.
\newblock {\em M{\'e}thodes num{\'e}riques et mod{\`e}les de sources
  a{\'e}roacoustiques fond{\'e}s sur l’{\'e}quation de {Galbrun}}.
\newblock PhD thesis, Universit{\'e} de Technologie de Compiegne, 2003.

\bibitem{Ga31}
Henri Galbrun.
\newblock {\em Propagation d'une onde sonore dans l'atmosphère et th{\'e}orie
  des zones de silence}.
\newblock Gauthier-Villars, Paris, 1931.

\bibitem{Go97}
Oleg~A. Godin.
\newblock Reciprocity and energy theorems for waves in a compressible
  inhomogeneous moving fluid.
\newblock {\em Wave Motion}, 25(2):143 -- 167, 1997.

\bibitem{Ha70}
V. D.~Hayes
\newblock Conservation of action and modal wave action.
\newblock {\em Proceedings of the Royal Society of London A: Mathematical,
  Physical and Engineering Sciences}, 320(1541):187--208, 1970.

\bibitem{KrLo89}
Heinz-Otto Kreiss and Jens Lorenz.
\newblock {\em Initial-boundary value problems and the Navier-Stokes
  equations}.
\newblock Academic Press, Boston, 1989.

\bibitem{Le03}
Guillaume Legendre.
\newblock {\em Rayonnement acoustique dans un fluide en {\'e}coulement: analyse
  math{\'e}matique et num{\'e}rique de l'{\'e}quation de {Galbrun}}.
\newblock PhD thesis, ENSTA ParisTech, 2003.

\bibitem{LyOs67}
D.~Lynden-Bell and J.~P. Ostriker.
\newblock On the stability of differentially rotating bodies.
\newblock {\em Monthly Notices of the Royal Astronomical Society},
  136(3):293--310, 1967.

\bibitem{MeBoMiPePe12}
Jean~Francois Mercier, Anne-Sophie Bonnet-Ben Dhia, Florence Millot, Sebastien
  Pernet, and Emilie Peynaud.
\newblock Time-harmonic acoustic scattering in a complex flow.
\newblock In {\em Acoustics 2012}, 2012.

\bibitem{Mo03}
Peter Monk.
\newblock {\em Finite element methods for {Maxwell's} equations}.
\newblock Clarendon, Oxford, 2003.

\bibitem{Ra85}
Jeffrey Rauch.
\newblock Symmetric positive systems with boundary characteristic of constant
  multiplicity.
\newblock {\em Transactions of the American Mathematical Society},
  291(1):167--187, 1985.

\bibitem{FrSc78a}
Bernard~F Schutz and John~L Friedman.
\newblock Langrangian perturbation theory of nonrelativistic fluids.
\newblock {\em The Astrophysical Journal}, 221:937--957, 1978.

\bibitem{TrGaBe03}
Fabien Treyssède, Gwénaël Gabard, and Mabrouk Ben~Tahar.
\newblock A mixed finite element method for acoustic wave propagation in moving
  fluids based on an {Eulerian--Lagrangian} description.
\newblock {\em The Journal of the Acoustical Society of America},
  113(2):705--716, 2003.

\end{thebibliography}

\end{document}